\documentclass[openright, 12pt]{article}
\usepackage{amsmath,amsfonts}
\usepackage{amsthm}
\usepackage{geometry}
\usepackage{float} 
\usepackage{multirow} 
\usepackage[dvips]{graphicx}
\usepackage{psfrag}

\theoremstyle{plain}

\theoremstyle{definition}

\newtheorem{innercustomthm}{Theorem}
\newenvironment{customthm}[1]
  {\renewcommand\theinnercustomthm{#1}\innercustomthm}
  {\endinnercustomthm}

\newtheorem{thm}{Theorem}
\newtheorem{prop}[thm]{Proposition}
\newtheorem{lem}[thm]{Lemma}
\newtheorem{rem}[thm]{Remark}
\newtheorem{cor}[thm]{Corollary}
\newtheorem{exa}[thm]{Example}
\newtheorem{defn}[thm]{Definition}
\newtheorem{con}[thm]{Conjecture}

\DeclareMathOperator{\Geo}{Geo}
\DeclareMathOperator{\OptGeo}{OptGeo}
\DeclareMathOperator{\Lip}{Lip}
\DeclareMathOperator{\LIP}{LIP}
\DeclareMathOperator{\Ch}{Ch}

\begin{document}

\begin{center}
 \Large On fundamental groups of $RCD$ spaces

\normalsize Jaime Santos-Rodr\'iguez, Sergio Zamora-Barrera

jaime.santsr@gmail.com, zamora@mpim-bonn.mpg.de\\

\end{center}

\vspace{-0.4cm}

\begin{abstract}
We obtain results about fundamental groups of $RCD^{\ast}(K,N)$ spaces previously known under additional conditions such as smoothness or lower sectional curvature bounds. For fixed $K \in \mathbb{R}$, $N \in [1,\infty )$, $D > 0 $, we show the following,
\begin{itemize}
    \item There is $C>0$ such that for each $RCD^{\ast}(K,N)$ space $X$ of diameter $\leq D$, its fundamental group $\pi_1(X)$ is generated by at most $C$ elements.
    \item There is $\tilde{D}>0$ such that for each $RCD^{\ast}(K,N)$ space $X$ of diameter $\leq D$ with compact universal cover $\tilde{X}$, one has diam$(\tilde{X})\leq \tilde{D}$.
    \item If a sequence of  $RCD^{\ast}(0,N)$ spaces $X_i$ of diameter $\leq D$ and rectifiable  dimension $n$ is such that their universal covers $\tilde{X}_i$ converge in the pointed Gromov--Hausdorff sense to a space $X$ of rectifiable dimension $n$, then there is $C>0$ such that for each $i$, the fundamental group $\pi_1(X_i)$ contains an abelian subgroup of index $\leq C$.
    \item If a sequence of  $RCD^{\ast}(K,N)$ spaces $X_i$ of diameter $\leq D$ and rectifiable dimension $n$ is such that their universal covers $\tilde{X}_i$ are compact and converge in the pointed Gromov--Hausdorff sense to a space $X$ of rectifiable dimension $n$, then there is $C>0$ such that for each $i$, the  fundamental group $\pi_1(X_i)$ contains an abelian subgroup of index $\leq C$.
    \item If a sequence of $RCD^{\ast}(K,N)$ spaces $X_i$ with first Betti number $\geq r$ and rectifiable dimension $ n$ converges in the Gromov--Hausdorff sense to a compact space $X$ of rectifiable dimension $m$, then the first Betti number of $X$ is at least $r + m - n$.
\end{itemize}
The main tools are the splitting theorem by Gigli, the splitting blow-up property by Mondino--Naber, the semi-locally-simple-connectedness of $RCD^{\ast}(K,N)$ spaces by Wang, the isometry group structure by Guijarro and the first author, and the structure of approximate subgroups by Breuillard--Green--Tao.
\end{abstract}


\section{Introduction}\label{sec:intro}

For $K \in \mathbb{R},$ $ N \in [1, \infty )$, the class of $RCD^{\ast}(K,N) $ spaces consists of proper  metric measure spaces that satisfy a synthetic condition of having Ricci curvature bounded below by $K$ and dimension bounded above by $N$. This class is closed under measured Gromov--Hausdorff convergence and contains the class of complete Riemannian manifolds of Ricci curvature $\geq K$ and dimension $\leq N$.

Building upon work of Mondino, Pan, Wei, and himself (\cite{mondino-wei}, \cite{pan-wang}, \cite{pan-wei}, \cite{wang}), Wang recently proved that $RCD^{\ast}(K,N)$ spaces are semi-locally-simply-connected \cite{wang-rcd}.  All the results in this paper were originally proved for the revised fundamental group, defined as the group of deck transformations of the universal cover, but thanks to the work of Wang, we can state them in terms of fundamental groups.

It was shown by Sosa, and independently by Guijarro and the first author that the isometry group of an arbitrary $RCD^{\ast}(K,N)$ space $(X,d,\mathfrak{m})$ is a Lie group. As noted by Cheeger--Colding \cite{cheeger-colding}, one could exploit this property to draw results about the groups $\pi_1(X)$. This is what we do throughout this paper.

$RCD^{\ast}(K,N)$ spaces have a well defined notion of dimension called \textit{rectifiable dimension} (see Theorem \ref{bs-dimension}), which is always an integer between 0 and $N$, and is lower semi-continuous with respect to pointed measured Gromov--Hausdorff convergence (see Theorem \ref{dim-semicont-2}).

For $K \in \mathbb{R},$ $ N \in [1, \infty ) $, $D>0$, we will denote by $RCD^{\ast}(K,N; D)$ the class of compact $RCD^{\ast}(K,N)$ spaces of diameter $\leq D$. We write $C(\alpha , \beta, \gamma )$ to denote a constant that depends only on the quantities $\alpha , \beta, \gamma$.

\subsection{Finite generatedness and compact universal covers}

For a compact semi-locally-simply-connected geodesic space $X$, its fundamental group is always finitely generated  (see \cite{mondello-mondino-perales}, Proposition 2.25). For the class of compact $N$-dimensional smooth Riemannian manifolds of Ricci curvature $\geq K$ and diameter $\leq D$, Kapovitch--Wilking showed that the number of generators of the fundamental group can be controlled in terms of $K$, $N$, and $D$ \cite{kapovitch-wilking}. We extend this result to the non-smooth case.

\begin{thm}\label{fg}
 Let $(X,d,\mathfrak{m})$ be an $RCD^{\ast}(K,N;D)$ space. Then $\pi_1(X)$ can be generated by $\leq C(N,K,D)$ elements.
\end{thm}

 In the particular case of spaces with a lower sectional curvature bound, there is an elementary proof of Theorem \ref{fg}, and one could get a simple explicit expression of $C$ in terms of $K,$ $ N,$ and  $D$ (see \cite{gromov-af}, Section 2). 

Kapovitch--Wilking showed that if a compact $N$-dimensional smooth Riemannian manifold has Ricci curvature $\geq K$, diameter $\leq D$, and finite fundamental group, the diameter of its universal cover is bounded above by a constant depending on $K$, $N$, and $D$ \cite{kapovitch-wilking}. We extend this result to the non-smooth case.

\begin{thm}\label{diam}
 Let $(X,d,\mathfrak{m})$ be an $RCD^{\ast}(K,N;D)$ space. If the universal cover $(\tilde{X},\tilde{d},\tilde{\mathfrak{m}})$ is compact, then diam$( \tilde{X} )\leq \tilde{D}( N,K, D)$.
\end{thm}

The proof of Theorem \ref{diam} is considerably different from the one for the smooth case, which relies heavily on the $C$-nilpotency of fundamental groups of manifolds with almost non-negative Ricci curvature, unknown for $RCD^{\ast}(K,N)$ spaces. Our proof of Theorem \ref{diam} is based on the Malcev embedding theorem (see Theorem \ref{malcev}), and the structure of approximate groups by Breuillard--Green--Tao \cite{breuillard-green-tao}.

\subsection{Uniform virtual abelianness}

From the Cheeger--Gromoll splitting theorem \cite{cheeger-gromoll}, one can conclude that the fundamental group of a compact Riemannian manifold of non-negative Ricci curvature is virtually abelian. It is not known whether the index of this abelian subgroup can be bounded by a constant depending only on the dimension. Mazur--Rong--Wang showed that for smooth Riemannian manifolds of non-negative sectional curvature and diameter $\leq 1$, this index depends only on the dimension and the volume of a ball of radius $1$ in the universal cover \cite{mazur-rong-wang}.  We generalize their results to non-smooth spaces with lower Ricci curvature bounds.

\begin{thm}\label{ua-1}
 Let $(X_i,d_i,\mathfrak{m}_i)$ be a sequence of $RCD^{\ast}(0,N;D)$ spaces of rectifiable dimension $ n$ such that the sequence of universal covers $(\tilde{X}_i,\tilde{d}_i,\tilde{\mathfrak{m}}_i)$ converges in the pointed measured Gromov--Hausdorff sense to some $RCD^{\ast}(0,N)$ space of rectifiable dimension $n$. Then there is $C > 0 $ such that for each $i$ there is an abelian subgroup $A_i \leq \pi_1(X_i)$ with index  $[\pi_1(X_i): A_i] \leq C$.
\end{thm}

\begin{thm}\label{ua-2}
 Let $(X_i,d_i,\mathfrak{m}_i)$ be a sequence of $RCD^{\ast}(K,N;D)$ spaces of rectifiable dimension $n$ with compact universal covers $(\tilde{X}_i,\tilde{d}_i,\tilde{\mathfrak{m}}_i)$. If the sequence $(\tilde{X}_i,\tilde{d}_i,\tilde{\mathfrak{m}}_i)$ converges in the pointed measured Gromov--Hausdorff sense to some $RCD^{\ast}(K,N)$ space of rectifiable dimension $n$, then there is $C > 0 $ such that for each $i$ there is an abelian subgroup $A_i \leq \pi_1(X_i)$ with index  $[\pi_1(X_i): A_i] \leq C$.
\end{thm}

The proofs of Theorems \ref{ua-1} and \ref{ua-2} are based on Turing and Kazhdan results on discrete approximations of Lie groups (\cite{turing}, \cite{kazhdan}). Just like in the smooth case, the key step consists on showing that non-collapsing sequences of $RCD^{\ast}(K,N)$ spaces do not admit non-trivial sequences of small groups of isometries (see Theorem \ref{nss}).

\subsection{First Betti number}

The second author has shown that when a sequence of smooth Riemannian manifolds collapses under a lower Ricci curvature bound, the first Betti number cannot drop more than the dimension \cite{zamora-betti}. We extend this result to the non-smooth setting.

\begin{thm}\label{b1}
 Let $(X_i,d_i,\mathfrak{m}_i)$ be $RCD^{\ast}(K,N;D)$ spaces of rectifiable dimension $ n$ and first Betti number $\beta_1(X_i) \geq r$. If the sequence $X_i$ converges in the measured Gromov--Hausdorff sense to an $RCD^{\ast}(K,N;D)$ space $X$ of rectifiable dimension $m$, then
\[\beta_1(X) \geq r + m - n .\] 
\end{thm}

\subsection{Previously known results}

Before it was known that $RCD^{\ast}(K,N)$ spaces are semi-locally-simply-connected,  Mondino--Wei showed that they always have universal covers. In that same paper they obtain some consequences parallel to classic results. We mention a couple of them here, and refer the reader to \cite{mondino-wei} for a more complete list. 

\begin{thm}
 (Myers) Let $(X,d,\mathfrak{m})$ be a $RCD^{\ast}(K,N)$ space with $K > 0$ and $N > 1$. Then $\pi_1(X)$ is finite.
\end{thm}

\begin{thm}
(Torus rigidity) Let $(X,d,\mathfrak{m})$ be an $RCD^{\ast}(0, N ; 1)$ space such that  $\pi_1(X) $ has $\lfloor N \rfloor$ independent elements of infinite order. Then a finite sheeted cover of $(X,d,\mathfrak{m})$ is isomorphic as a metric measure space to a flat Riemannian torus.
\end{thm}

\begin{thm}
 (Milnor polynomial growth) Let $X$ be a $RCD^{\ast}(0, N)$ space. Then any finitely generated subgroup of $\pi_1(X)$ is virtually nilpotent.
\end{thm}

In a more recent work Mondello--Mondino--Perales obtained a first Betti number upper bound and a classification of spaces attaining it \cite{mondello-mondino-perales}.

\begin{thm}
 (Torus stability) Let $(X,d, \mathfrak{m})$ be an $RCD^{\ast}(K,N;D)$ space with $\beta_1(X) \geq \lfloor N \rfloor $ and $KD^2 \geq - \varepsilon (N)$. Then 
\begin{itemize}
\item The rectifiable dimension of $X$ is $\lfloor N \rfloor$, and $X$ is $\lfloor N \rfloor $-rectifiable.
\item A finite sheeted cover of $X$ is $\delta ( KD^2 )$-close in the Gromov--Hausdorff sense to a flat torus with $\delta \to 0$ as $ KD^2  \to 0$.
\item If $N \in \mathbb{N}$, then $\mathfrak{m}$ is a constant multiple of the $ N $-dimensional Hausdorff measure, and $X$ is bi-H\"older homeomorphic to a flat Riemannian torus.
\end{itemize}
\end{thm}

From the work of Breuillard--Green--Tao, one can conclude that the fundamental group of almost non-negatively curved $RCD^{\ast}(K,N)$ spaces is virtually nilpotent \cite{breuillard-green-tao}. 

\begin{thm}
 (Virtual nilpotency) There is $\varepsilon (N) >0$, such that for any $RCD^{\ast}(- \varepsilon , N;1) $ space $(X,d,\mathfrak{m})$ of rectifiable dimension $m$, there is a subgroup $G \leq \pi_1(X)$ and a finite normal subgroup $H \triangleleft G$ such that
\begin{itemize}
\item $[ \pi_1(X) : G ] \leq C(N)$.
\item $G/H$ is isomorphic to a lattice in a nilpotent Lie group of dimension $\leq m$. 
\end{itemize}
In particular, $\beta_1(X)\leq m$.
\end{thm}

\subsection{Open problems}

Naturally, it would be interesting to know if the fundamental group of a space of non-negative Ricci curvature is finitely generated.

\begin{con}
\rm (Milnor) Let $(X,d,\mathfrak{m})$ be an $RCD^{\ast}(0,N)$ space. Then $\pi_1(X)$ is finitely generated.
\end{con}

It has been pointed out multiple times (see \cite{fukaya-yamaguchi}, \cite{kapovitch-petrunin-tuschmann}, \cite{kapovitch-wilking}, \cite{mazur-rong-wang}) that an important problem is to determine if one could remove the non-collapsing hypothesis from Theorems  \ref{ua-1} and \ref{ua-2}.

\begin{con}
\rm (Fukaya--Yamaguchi) Let $(X,d,\mathfrak{m})$ be an $RCD^{\ast}(0,N;1)$ space. Then $\pi_1(X)$ contains an abelian subgroup $A$ of index $[\pi_1(X): A] \leq C(N)$.
\end{con}

\begin{con}
\rm (Fukaya--Yamaguchi) Let $(X,d,\mathfrak{m})$ be an $RCD^{\ast}(K,N;D)$ space. If $\pi_1(X)$ is finite, then it contains an abelian subgroup $A$ of index $[\pi_1(X): A] \leq C(N, KD^2)$.
\end{con}

It is well known that in the class of $RCD^{\ast}(K,N;D)$ spaces, the first Betti number is lower semi-continuous with respect to Gromov--Hausdorff convergence (see Theorem \ref{sw}). It would be interesting to determine if this still holds for their universal covers.

\begin{con}
\rm Let $(X_i,d_i,\mathfrak{m}_i)$ be a sequence of $RCD^{\ast}(K,N;D)$ spaces and assume their universal covers $(\tilde{X}_i,\tilde{d}_i,\tilde{\mathfrak{m}}_i)$ converge in the pointed measured Gromov--Hausdorff sense to an $RCD^{\ast}(K,N)$ space $X$. Then $(X,d,\mathfrak{m})$ is simply connected.
\end{con}

\subsection{Overview}

In Section \ref{section-2}, we introduce the definitions and background material we will need. In Sections 2.2, 2.3, we review the definition and  properties of $RCD^{\ast}(K,N)$ spaces. In Sections 2.4, 2.5, we give the definition  of Gromov--Hausdorff convergence of metric measure spaces and equivariant convergence of groups of isometries. In Section 2.6 we recall the short basis construction. In Section 2.7 we gather results concerning approximations of Lie groups by discrete groups. Finally in Section 2.8 we discuss the relationship between these discrete approximations and equivariant convergence of groups.

In Section \ref{section-3}, we prove Theorem \ref{fg}. Our proof is not significantly different in spirit from the one by Kapovitch--Wilking for smooth spaces (see Theorem \ref{short-short}). Assuming the theorem fails one gets a sequence $(X_i,d_i,\mathfrak{m}_i)$ of $RCD^{\ast}(K,N)$ spaces contradicting the statement of the theorem. Then by taking a subsequence one can assume the sequence $(X_i,d_i,\mathfrak{m}_i)$ converges to some space of dimension $m$, and proceed by reverse induction on $m$. The induction step consists of blowing up a contradictory sequence via a construction of Mondino--Naber in order to obtain another one that converges to a space of strictly higher dimension.

In Section \ref{section-4}, we prove Theorem \ref{diam}. Assuming the Theorem fails one gets a sequence of $RCD^{\ast}(K,N;D)$ spaces $X_i$ whose universal covers $(\tilde{X}_i,\tilde{d}_i,\tilde{\mathfrak{m}}_i)$ are compact but their diameters go to infinity. After taking a subsequence one can assume that the sequence $(\tilde{X}_i,\tilde{d}_i,\tilde{\mathfrak{m}}_i)$ converges to a non-compact space $(X,d,\mathfrak{m})$ and the actions of $\pi_1(X_i)$ converge to the action of a non-compact Lie group of isometries $\Gamma \leq Iso(X)$ (see Definition \ref{def:equivariant}). From this we can extract a sequence of approximate morphisms $\pi_1(X_i) \to \Gamma$ (see Definition \ref{def:discreteapprox}). This ends up being incompatible since the groups $\pi_1(X_i)$ are finite but $\Gamma$ is not compact (see Theorem \ref{infinite}).

In Section \ref{section-5}, we show Theorems \ref{ua-1} and \ref{ua-2}. The key step is showing that when a sequence of $RCD^{\ast}(K,N)$ spaces doesn't collapse, it cannot have small groups of measure preserving isometries (see Theorem \ref{nss}). The results then follow from group theory (see Theorems \ref{turing} and \ref{Turing-no-compact}). 

In Section \ref{section-6}, we prove Theorem \ref{b1}. We first extend the normal subgroup theorem by Kapovitch--Wilking to the non-smooth setting. The normal subgroup theorem provides a subgroup of the first homology group that can be detected essentially everywhere (see Theorem \ref{normal-abelian}). We then apply a Theorem of Sormani--Wei that allows us to compute the first homology of the limit space $(X,d,\mathfrak{m})$ in terms of the first homology groups of the spaces $(X_i,d_i,\mathfrak{m}_i)$ and their subgroups generated by small loops (see Theorem \ref{sw} and Corollary \ref{sw-gap}).

\section{Preamble}\label{section-2}

\subsection{Notation}\label{notation-subsection}

If $(X,d)$ is a metric space, $p \in X$, and $r > 0 $, we denote the open ball of radius $r$ around $p$ in $X$ by $B_p(r,X)$.  For $m \in  \mathbb{N}$, the $m$-dimensional Hausdorff measure is denoted by $\mathcal{H}^m$, which we assume normalized so that 
\[ \int_{B_0(1,\mathbb{R}^m)}(1- d(0,\cdot )) d\mathcal{H}^m  = 1   . \]
For a metric space $(X,d)$, we can adjoin a point $\ast$ at infinite distance from any point of $X$ to get a new space we denote as $X \cup \{ \ast \}$. Similarly, to any group $G$ we can adjoin an element $\ast$, whose product with any element of $G$ is defined as $\ast$, obtaining an algebraic structure on $G \cup \{ \ast \}$. 

For a group $G$ and elements $g_1, \ldots , g_n \in G$, we denote by $g_1 \ldots g_n$ the word whose letters are the  $g_i$'s, and by $g_1 \cdots g_n \in G$ the result of multiplying the elements $g_i$. For a subset $S \subset G$, and $n \in \mathbb{N}$, we write $S^n : = \{ s_1 \cdots s_n \in G \vert s_1, \ldots , s_n \in S \}$.

We write $C(\alpha , \beta ,\gamma )$ to denote a constant $C$ that depends only on the quantities $\alpha , \beta , \gamma $.

\subsection{Optimal transport}

In this section we introduce the basic theory of metric measure spaces we will need. We refer the reader to \cite{ambrosio-gigli} for proofs and further details.

 A \textit{metric measure space} $(X,d,\mathfrak{m})$ consists of a complete, separable, geodesic metric space $(X,d)$ and a Radon measure $\mathfrak{m}$ on the Borel $\sigma$-algebra $\mathcal{B}(X)$ with supp$(\mathfrak{m}) = X$ and $\mathfrak{m}(X) > 0$.  We denote by $\mathbb{P}(X)$ the set of probability measures on $\mathcal{B}(X)$ and by
$$\mathbb{P}_2(X) := \left\lbrace \mu \in \mathbb{P}(X) \,|\,  \int_X d^2(x_0, \cdot )d\mu < \infty \text{ for some (and hence all) } x_0 \in X   \right\rbrace $$
 the set of probability measures with finite second moments.

Given measures $\mu_1, \mu_2 \in \mathbb{P}_2(X)$, a \textit{coupling} between them  is a probability measure $\pi \in \mathbb{P}(X\times X)$ with $p_{1\#}\pi = \mu_1 $ and $p_{2\#}\pi = \mu_2 $,  where $p_1,p_2 : X\times X \rightarrow X,$  are the natural projections.  The $L^2$\textit{-Wassertein distance} in $\mathbb{P}_2(X)$ is defined   as:
 $$\mathbb{W}_2^2 (\mu,\nu) := \inf\left\lbrace \int_{X\times X} d^2(x,y) d\pi(x,y) \,|\, \pi \text{ is a coupling between } \mu \text{ and } \nu \right\rbrace. $$ 
 It is known that  $(\mathbb{P}_2(X),\mathbb{W}_2)$ inherits good properties from $(X,d)$ such as being complete, separable and geodesic.

We denote by $\Geo(X)$ the space of constant speed geodesics $\gamma : [0,1] \to X$  and equip it with the topology of uniform convergence. For $t \in [0,1]$ we denote the evaluation map $e_t : \Geo(X) \rightarrow X$ as $\gamma \mapsto \gamma_t.$ 

Given $\mu_0, \mu_1 \in \mathbb{P}_2(X)$ and a geodesic $\left(\mu_t \right)_{t \in [0,1]}$ in $\mathbb{P}_2(X),$  there is $\pi \in \mathbb{P}(\Geo(X))$ with $e_{t\#}\pi = \mu_t$ for all $t \in [0,1]$, and 
\begin{equation}\label{opt}
    \mathbb{W}_2^2 (\mu_0,\mu_1) = \int_{\Geo(X)}\text{length}^2(\gamma)d\pi(\gamma) .
\end{equation} 
The collection of all measures $\pi \in \mathbb{P}( \Geo (X))$ with $e_{0\#}\pi =\mu_0$, $e_{1\#}\pi = \mu_1$, and satisfying Equation \ref{opt} with will be denoted as $\OptGeo(\mu_0,\mu_1) $.
 \subsection{Riemannian curvature dimension condition}
 In \cite{ambrosio-gigli-savare}, Ambrosio--Gigli--Savar\'e introduced a class of  metric probability spaces satisfying a  ``Riemannian curvature dimension'' condition. This class was later extended by Ambrosio--Gigli--Mondino--Rajala to obtain what is known now as $RCD^{\ast}(K,N)$ spaces \cite{agmr}. Multiple equivalent definitions and reformulations have been obtained throughout the years. We give a definition here but refer the reader to (\cite{ambrosio-gigli-savare}, \cite{erbar-kuwada-sturm}) for many others.

For $K \in \mathbb{R},$  $N \in [1, \infty)$,  the distortion coefficients  $\sigma_{K,N}^{(\cdot )}(\cdot):[0,1]\times \mathbb{R}^{+}$ are defined as  
\begin{equation*}
\sigma_{K,N}^{(t)}(\theta):=
\begin{cases}
\infty & \text{if }\, K\theta^2\geq N\pi^2, \\
\frac{\sin(t\theta \sqrt{K/N})}{\sin(\theta\sqrt{K/N})} & \text{if }\, 0<K\theta^2<N\pi^2,\\
t & \text{if }\, K\theta^2=0,\\
\frac{\sinh(t\theta\sqrt{-K/N})}{\sinh(\theta\sqrt{-K/N})} & \text{if } K\theta^2<0.
\end{cases}
\end{equation*}
\begin{defn}
\rm For $K \in \mathbb{R}$, $N \geq 1$, we say that a metric measure space $(X,d,\mathfrak{m})$ is a $CD^{\ast}(K,N)$ space if for any pair of measures $\mu_0, \mu_1 \in \mathbb{P}_2(X)$ with bounded support there is a measure $\pi \in \OptGeo (\mu_0, \mu_1) $ such that for each $t \in [0,1]$ and $N^{\prime} \geq N$, one has 
\[ \int_X \rho_t ^{1-\frac{1}{N^{\prime}}}  d \mathfrak{m} \geq \int_{\Geo (X)}\left( \sigma_{K,N^{\prime}}^{(1-t)} (d(\gamma_0, \gamma_1))\rho_0^{-\frac{1}{N^{\prime}}} (\gamma _0) + \sigma_{K,N^{\prime}}^{(t)} (d(\gamma_0, \gamma_1))\rho_1^{-\frac{1}{N^{\prime}}} (\gamma _1) \right)   d \pi (\gamma)  , \]
where $ e_{t\#}\pi = \rho_t \mathfrak{m} + \mu_t^{s}$ with  $\mu_t^{s} \perp \mathfrak{m}$ for each $t \in [0,1]$. 
\end{defn}
We denote by $\LIP (X)$ the set of Lipschitz functions $f: X \to \mathbb{R}.$ For  $f  \in \Lip(X)$, its \textit{local Lipschitz constant} $\Lip f(x) : X \to \mathbb{R} $ is defined as:
\[ \Lip f(x) := \limsup_{y\rightarrow x} \frac{|f(x)-f(y) |}{d(x,y)}. \]

\begin{defn}\label{def.Cheegerenergy}
\rm Let  $(X,d,\mathfrak{m})$ be a metric measure space.  For   $f \in L^2(\mathfrak{m})$ we define its \textit{Cheeger energy} as:
\[  \Ch_{\mathfrak{m}} (f) :=\inf \left\lbrace \liminf_{n \rightarrow \infty} \frac{1}{2}\int |\Lip f_n|^2 d \mathfrak{m} \,|\, f_n \in \LIP(X),  f_n \rightarrow f  \text{ in } L^2 (\mathfrak{m})   \right\rbrace. \]
\end{defn}
The \textit{Sobolev space} $W^{1,2}(X,d,\mathfrak{m})$ is defined as the space of $L^2$-functions $f:X \to \mathbb{R}$ with $\Ch_{\mathfrak{m}}(f) < \infty$. It is equipped with the norm
\[ \Vert f \Vert _{W^{1,2}}:= \Vert f\Vert_{L^2(\mathfrak{m})} + 2 \Ch_{\mathfrak{m}}(f).\]
If this space is a Hilbert space we say that $(X,d,\mathfrak{m})$ is \textit{infinitesimally Hilbertian}.

\begin{defn}
\rm We say a $CD^{\ast}(K,N)$ space $(X,d,\mathfrak{m})$ is an $RCD^{\ast}(K,N)$ space if it is infinitesimally Hilbertian.
\end{defn}
A well known property of the $RCD^{\ast}(K,N)$ condition is that it can be checked locally. We refer the reader to (\cite{erbar-kuwada-sturm}, Section 3) by Erbar--Kuwada--Sturm for details and further  discussions.

\begin{thm}\label{local-to-global}
 Let $(X,d , \mathfrak{m}) $ be a metric measure space. If for each $x \in X$ there is $r > 0 $, a pointed $RCD^{\ast}(K,N)$ space $(Y,d^Y, \mathfrak{m}^Y,y)$, and a  measure preserving isometry $\varphi_x: B_x(r,X)\to B_y(r,Y)$, then $(X,d,\mathfrak{m})$ is an $RCD^{\ast}(K,N)$ space. 
\end{thm}

\begin{proof}
Since $(Y,d^Y,\mathfrak{m}^Y)$ is an $RCD^{\ast}(K,N)$ space, it
satisfies $CD_{loc}^e(K,N)$ (see (\cite{erbar-kuwada-sturm}, Section 3) for the precise definition).  Hence via $\varphi_x$ 
 we can find a neighborhood of $x$ such that for any two measures supported 
there and absolutely continuous with respect to $\mathfrak{m}$, there is a Wasserstein geodesic between them satisfying the inequality 
of the $CD^e(K,N)$ condition (see \cite{erbar-kuwada-sturm}, Equation 3.1). Therefore, from (\cite{erbar-kuwada-sturm}, Theorem 3.14),  $(X,d,\mathfrak{m})$ is a $CD^e(K,N)$ space.

Similarly, we can find for every $x\in X$ via $\varphi_x$ a ball $B$ with $\mathfrak{m}(\partial B)=0$ and such that $(B,d|_B,\mathfrak{m}\llcorner B)$ is 
infinitesimally Hilbertian. From (\cite{ambrosio-gigli-savare}, Theorem 6.22) it then follows that $(X,d,\mathfrak{m})$ is infinitesimally Hilbertian.

Finally by (\cite{erbar-kuwada-sturm}, Theorem 7), this is equivalent to being an $RCD^{\ast}(K,N)$ space.    
\end{proof}

\begin{rem}\label{normalize}
\rm A direct computation shows that if $(X,d,\mathfrak{m}) $ is an $RCD^{\ast}(K,N)$ space, then for any $c>0$, $(X,d,c \mathfrak{m})$ is also an $RCD^{\ast} (K,N) $ space, and for any $\lambda > 0$, $(X, \lambda d , \mathfrak{m})$ is an $RCD^{\ast}(\lambda ^{-2}K,N)$ space.
\end{rem}

Let $(X,d^X,\mathfrak{m})$ be an $RCD^{\ast}(K,N)$ space  and $\rho : Y \to X$ a covering space. $Y$ has a natural geodesic structure such that for any curve $\gamma : [0,1] \to Y$ one has 
\[ \text{length}(\rho \circ \gamma) = \text{length}(\gamma).\]
 Throughout this paper, we will implicitly assume $Y$ carries such geodesic metric. Set
\[  \mathcal{W} : = \{ W \subset Y  \text{ open bounded }\vert \text{ } \rho \vert_{W} : W \to \rho (W)  \text{ is an isometry}    \}     \] 
and define a measure $\mathfrak{m}^{Y}$ on $Y$ by setting $\mathfrak{m}^Y(A) : = \mathfrak{m}(\rho (A))$ for each Borel set $A$ contained in an element of $\mathcal{W}$. The measure $\mathfrak{m}^Y$ makes $\rho : Y \to X$ a local isomorphism of metric measure spaces, so by Theorem \ref{local-to-global}  $(Y,d^Y,\mathfrak{m}^Y)$ is an $RCD^{\ast}(K,N)$ space, and its group of deck transformations  acts by measure preserving isometries. In particular, this holds for the universal cover  \cite{wang}.  

\begin{thm}\label{jikang}
 (Wang) Let $(X,d,\mathfrak{m})$ be an $RCD^{\ast}(K,N)$ space. Then $X$ is semi-locally-simply-connected, so its universal cover $\tilde{X}$ is simply connected.
\end{thm}

Due to Theorem \ref{jikang}, for an $RCD^{\ast}(K,N)$ space $(X,d,\mathfrak{m})$ we can think of its fundamental group $\pi_1(X)$ as the group of deck transformations of the universal cover $\tilde{X}$. 

One of the most powerful tools in the study of $RCD^{\ast}(K,N)$ spaces is the Bishop--Gromov inequality by Bacher--Sturm. All we will need here is the following version, and refer the reader to \cite{bacher-sturm} for stronger ones.

\begin{thm}\label{bg-in}
 (Bishop--Gromov) For each $K \in \mathbb{R}$, $N \geq 1$, $R>r>0$, there is $C(K,N, R, r)> 0 $ such that for any pointed $RCD^{\ast}(K,N)$ space $(X, d, \mathfrak{m},p)$, one has
\[\mathfrak{m}(B_p(R,X)) \leq C \cdot \mathfrak{m}(B_p(r,X)) . \]
\end{thm} 

Theorem \ref{bg-in} implies that each $RCD^{\ast}(K,N)$ is proper. For a proper metric space $X$, the topology that we use on its group of isometries $Iso(X)$ is the compact-open topology, which in this setting coincides with both the topology of pointwise convergence and the topology of uniform convergence on compact sets. This topology makes $Iso(X)$ a locally compact second countable metric group. In the case $(X,d,\mathfrak{m})$ is an $RCD^{\ast}(K,N)$ space, Sosa, and independently Guijarro with the first author showed that this group is a Lie group (\cite{guijarro-santos}, \cite{sosa}).

\begin{thm}\label{gs-lie}
 (Sosa, Guijarro--Santos) Let $(X,d,\mathfrak{m})$ be an $RCD^{\ast}(K,N)$ space. Then $Iso(X)$ is a Lie group.
\end{thm}

An important ingredient in the proof of Theorem \ref{gs-lie} is the following lemma \cite{guijarro-santos}.

\begin{lem}\label{no-fixed}
 Let $(X,d,\mathfrak{m})$ be an $RCD^{\ast}(K,N)$ space and $f: X \to X$ an isometry. If $\mathfrak{m} \left( \{ x \in X \vert fx=x \}     \right) \neq 0$, then $f = Id_X$.
\end{lem}

 Recall that if $X$ is a proper geodesic space and $\Gamma \leq Iso (X)$ is a closed group of isometries, the metric $d^{\prime}$ on $X/\Gamma$ defined as $d^{\prime}([x],[y]) : = \inf _{g \in \Gamma }(d(gx,y)) $ makes it a proper geodesic space. 
 
 Notice that for a metric measure space $(X,d,\mathfrak{m})$, the subgroup of $Iso (X)$ consisting of measure preserving transformation is closed.  By the work of Galaz--Kell--Mondino--Sosa, the class of $RCD^{\ast}(K,N)$ spaces is closed under quotients by compact groups of measure preserving isometries \cite{galaz-kell-mondino-sosa}.

\begin{thm}\label{compact-quotient}
 (Galaz--Kell--Mondino--Sosa) Let $(X,d,\mathfrak{m})$ be an $RCD^{\ast}(K,N)$ space and $\Gamma \leq Iso (X)$ a compact group of measure preserving isometries. Then the metric measure space $(X/\Gamma, d^{\prime}, \mathfrak{m}^{\prime} ) $ is an $RCD^{\ast}(K,N)$ space, where $\mathfrak{m}^{\prime}$ is the pushforward of $\mathfrak{m}$ under the projection $X \to X/ \Gamma$.
\end{thm}




For a proper metric space $X$, a group of isometries $\Gamma \leq Iso (X)$, and $p\in X$, the \textit{isotropy group at} $p$ is defined as the subgroup of $\Gamma $ consisting of the elements $g\in \Gamma$ with $g p = p$.  Recall that $\Gamma $ is discrete with respect to the compact-open topology if and only if its action on $X$ has discrete orbits and is almost free (the isotropy group at $x$ is finite for all $x \in X$). Using Theorem \ref{compact-quotient} one can show that the class  of $RCD^{\ast}(K,N)$ spaces is closed under quotients by discrete groups (\cite{galaz-kell-mondino-sosa}, Theorem $7.24$).

\begin{cor}\label{quotient}
 (Galaz--Kell--Mondino--Sosa)  Let $(X,d,\mathfrak{m})$ be an $RCD^{\ast}(K,N)$ space and $\Gamma \leq Iso (X)$ a discrete group of measure preserving isometries. Then the proper geodesic space $X/\Gamma$ admits a measure that makes it an $RCD^{\ast}(K,N)$ space.
\end{cor}

Due to Theorem \ref{compact-quotient} (Corollary \ref{quotient}), whenever we have an $RCD^{\ast}(K,N)$ space and a compact (discrete) group of measure preserving isometries $\Gamma \leq Iso (X)$, we will assume that the quotient $X / \Gamma$ is equipped with a measure that makes it an $RCD^{\ast}(K,N)$ space. However, notice that if $\Gamma$ is infinite and discrete then this measure is not the one given by the image of the measure of $X$ under the projection $X \to X/\Gamma$.

For $K \in \mathbb{R},$ $ N \in [1, \infty ) $, $D>0$, we will denote by $RCD^{\ast}(K,N; D)$ the class of compact $RCD^{\ast}(K,N)$ spaces of diameter $\leq D$. 

\subsection{Gromov--Hausdorff topology}

\begin{defn}\label{pmgh}
\rm Let $(X_i,p_i)$ be a sequence of pointed proper metric spaces. We say that it \textit{converges in the pointed Gromov--Hausdorff sense} (or pGH sense) to a proper pointed metric space $(X,p)$ if there is a sequence of functions $\phi_i : X_i \to X \cup \{ \ast \}$ with $\phi_i(p_i ) \to p $  and such that for each $R>0$ one has 
\begin{equation}\label{pgh2}
\phi_i^{-1}(B_p(R,X)) \subset B_{p_i} (2R, X_i) \text{ for large enough }i,
\end{equation}
\begin{equation}\label{pgh3}
\lim_{i \to \infty } \sup_{x_1,x_2 \in B_{p_i}(2R,X_i)} \vert d(\phi_i(x_1),\phi_i(x_2)) - d(x_1,x_2) \vert =0  , 
\end{equation}
\begin{equation}\label{pgh4}
  \lim_{i \to \infty} \sup_{y \in B_p(R,X) } \inf_{x \in B_{p_i}(2R,X_i)}  d(\phi _i (x),y)  = 0. 
\end{equation}
If in addition to that, $(X_i,d_i,\mathfrak{m}_i)$, $(X,d,\mathfrak{m})$ are metric measure spaces, the maps $\phi_i$ are Borel measurable, and
\[  \int_X f \cdot d((\phi_i)_{\ast}\mathfrak{m}_i) \to \int_X f \cdot d(\mathfrak{m})     \]
for all $f: X \to \mathbb{R}$ continuous with compact support, then we say that $(X_i,d_i, \mathfrak{m}_i,p_i)$ converges to $(X,d,\mathfrak{m},p)$ in the \textit{pointed measured Gromov--Hausdorff sense} (or pmGH sense).
\end{defn}

\begin{rem}
\rm Whenever we say that a sequence of spaces $X_i$ converges in the pointed (measured) Gromov--Hausdorff sense to some space $X$, we implicitly assume the existence of the maps $\phi_i$ satisfying the above conditions, and if a sequence $x_i \in X_i$ is such that $\phi_i (x_i) \to x \in X$, by an abuse of notation we say that $x_i$ \textit{converges} to $x$.
\end{rem}

\begin{rem}
\rm If there is a sequence of groups $\Gamma_i$ acting on $X_i$ by (measure preserving) isometries with diam$(X_i/\Gamma_i) \leq C$ for some $C>0$, one could ignore the points $p_i$ when one talks about pointed (measured) Gromov--Hausdorff convergence, as any pair of limits are going to be isomorphic as metric (measure) spaces. 

A particular instance of this situation is when all the spaces $X_i$ have diameter $\leq C$ for some $C>0$. In that case, we simply say that the sequence $X_i$ converges in the (measured) Gromov--Hausdorff sense (or (m)GH sense) to $X$.
\end{rem}

\begin{rem}\label{ghinv}
\rm  If a sequence of pointed proper metric spaces $(X_i, p_i)$ converges in the pGH sense to the pointed proper metric space $(X,p)$,  from the functions $\phi_i : X_i \to X \cup \{ \ast \}$ given by Definition \ref{pmgh} one could construct approximate inverses $\psi_i: X \to X_i$ such that for each $R>0$ one has 
\begin{equation}\label{ghinverse}
\sup_{x \in B_p(R,X)} d( \phi_i \psi_i (x) , x)  \to 0 \text{ as } i \to \infty   .
\end{equation}
\end{rem}
One of the main features of the class of $RCD^{\ast}(K,N)$ spaces is the Gromov--Hausdorff compactness property (\cite{gromov-ms}, \cite{bacher-sturm}).
\begin{defn}
\rm A pointed $RCD^{\ast}(K,N)$ space $(X,d,\mathfrak{m},p)$ is said to be \textit{normalized} if 
\[   \int_{B_p(1,X)}(1-d(p,\cdot ))d\mathfrak{m} = 1.  \]
Clearly there is a unique $c> 0$ such that $(X,d,c\mathfrak{m},p)$ is normalized, and by Remark \ref{normalize}, it is also an $RCD^{\ast}(K,N)$ space.
\end{defn}

\begin{thm}\label{compactness}
 (Gromov) If $(X_i,d_i, \mathfrak{m}_i,p_i)$ is a sequence of pointed normalized $RCD^{\ast}(K,N)$ spaces, then one can find a subsequence that converges in the pmGH sense to some pointed metric measure space $(X,d, \mathfrak{m},p)$.
\end{thm}

\begin{thm}\label{compactness-2}
 (Bacher--Sturm) The class of pointed normalized $RCD^{\ast}(K,N) $ spaces is closed under pmGH convergence.
\end{thm}

 Let $(X,d,\mathfrak{m})$ be an $RCD^{\ast}(K,N)$ space, $p \in X$, and $r >0$. We will denote as $\mathfrak{m}^{p}_{r}$ the constant multiple of $\mathfrak{m}$ that satisfies 
 \[  \int_{B_{p}(r, X)} \left( 1 - \frac{1}{r} d(p, \cdot )  \right) d \mathfrak{m}^{p}_r  = 1.  \] 
 That is, the pointed metric measure space $(X, r^{-1} d , \mathfrak{m}^{p}_r,p)$ is normalized.

\begin{defn}
\rm Let $(X,d,\mathfrak{m})$ be an $RCD^{\ast}(K,N)$ space and $m \in \mathbb{N}$. We say that $p \in X$ is an $m$\textit{-regular} point if for each $\lambda_i \to \infty$, the sequence $(X,\lambda_i d, \mathfrak{m}^{p}_{1/ \lambda_i}, p)$ converges in the pmGH sense to $(\mathbb{R}^m, d^{\mathbb{R}^m}, \mathcal{H}^m, 0)$.
\end{defn}
 
Mondino--Naber showed that the set of regular points in an $RCD^{\ast}(K,N)$ space has full measure \cite{mondino-naber}. This result was refined by Bru\'e--Semola who showed that most points have the same local dimension \cite{brue-semola}.

\begin{thm}\label{bs-dimension}
 (Bru\'e--Semola) Let $(X,d,\mathfrak{m})$ be an $RCD^{\ast}(K,N)$ space. Then there is a unique $m \in \mathbb{N} \cap [0, N]$ such that the set of $m$-regular points in $X$ has full measure. This number $m$ is called the \textit{rectifiable dimension }of $X$.
\end{thm}

\begin{rem}\label{homogeneous}
\rm By Theorem \ref{gs-lie}, an $RCD^{\ast}(K,N)$ space $(X,d,\mathfrak{m}) $ whose isometry group acts transitively is isometric to  a Riemannian manifold, so its rectifiable dimension coincides with its topological dimension.
\end{rem}

The well known Cheeger--Gromoll splitting theorem was extended by Cheeger--Colding to  limits of Riemannian manifolds with lower Ricci curvature bounds \cite{cheeger-colding}, and later by Gigli to the setting  of $RCD^{\ast}(K,N)$ spaces \cite{gigli}.

\begin{thm}\label{gigli}
 (Gigli) For each $i\in \mathbb{N}$, let $(X_i,d_i, \mathfrak{m}_i, p_i)$ be a pointed normalized $RCD^{\ast} (-\varepsilon_i ,N)$ space with $\varepsilon _ i \to 0$. Assume the sequence $(X_i, d_i,\mathfrak{m}_i, p_i)$ converges in the pmGH sense to a space $(X,d,\mathfrak{m},p)$. If $(X,d)$ contains an isometric copy of $\mathbb{R}^{m}$, then there is $c>0$ and a metric measure space $(Y, d^Y, \nu )$ such that $(X,d,c \mathfrak{m})$ is isomorphic to the product $(\mathbb{R}^m\times Y, d ^{\mathbb{R}^m} \times d^Y, \mathcal{H}^m \otimes \nu)$. Moreover, if $N-m \in [0,1)$ then $Y$ is a point, and in general, $(Y,d^Y,\nu)$ is an $RCD^{\ast}(0,N-m)$ space. 
\end{thm}

This condition of Theorem \ref{gigli} is not stable under blow-up. That is, if one considers a sequence $\lambda_i \to \infty$, then the sequence of normalized spaces $(X_i, \lambda_i d_i,  (\mathfrak{m}_i)^{p_i}_{1/\lambda_i} , p_i )$ may fail to converge to a space containing a copy of $\mathbb{R}^m$. However, Mondino--Naber showed that this issue could be fixed by slightly shifting the basepoint \cite{mondino-naber}.

\begin{thm}\label{zoom}
 (Mondino--Naber) For each $i\in \mathbb{N}$, let $(X_i,d_i, \mathfrak{m}_i)$ be a normalized $RCD^{\ast} (-\varepsilon_i ,N)$ space with $\varepsilon_i \to 0$. Assume that for some choice $p_i \in X_i$, the sequence $(X_i, d_i,\mathfrak{m}_i, p_i)$ converges in the pmGH sense to $(\mathbb{R}^m, d^{\mathbb{R}^m}, \mathcal{H}^m, 0)$. Then there is a sequence of subsets $U_i \subset B_{p_i}(1,X_i)$ with $\mathfrak{m}_i(U_i)/ \mathfrak{m}_i(B_{p_i}(1,X_i)) \to 1 $ such that for any sequence $\lambda_i \to \infty$ and any sequence $q_i \in U_i$, the sequence of spaces $(X_i, \lambda_id_i,(\mathfrak{m}_i)^{q_i}_{1/\lambda_i}, q_i)$ converges (up to subsequence) in the pmGH sense to a pointed $RCD^{\ast}(0,N)$ space containing an isometric copy of $\mathbb{R}^m$.
\end{thm}

Some consequences of Theorems \ref{gigli} and \ref{zoom} are the following results, which were proven by Kitabeppu in \cite{kitabeppu}.

\begin{thm}\label{dim-semicont}
 (Kitabeppu) Let $(X_i, d_i, \mathfrak{m}_i, p_i)$ be a sequence of pointed $RCD^{\ast}(K,N) $ spaces of rectifiable dimension $m$ converging in the pmGH sense to the space $(X,d, \mathfrak{m}, p)$. Then the rectifiable dimension of $X$ is at most $m$.
\end{thm}

\begin{cor}\label{regular-dimension}
 Let $(X,d,\mathfrak{m})$ be an $RCD^{\ast}(K,N)$ space of rectifiable dimension $m$. Then 
\[  m = \sup \{ n \in \mathbb{N} \vert \text{ there is an }n\text{-regular point }x \in X   \}. \]
\end{cor}

\begin{rem}\label{dim-semicont-2}
\rm From Theorem \ref{dim-semicont}, it follows that if the spaces $X_i$ in Theorem \ref{gigli} have rectifiable dimension $n$, then the rectifiable dimension of the space $Y$ is at most $n-m$.
\end{rem}

\subsection{Equivariant Gromov--Hausdorff convergence}

There is a well studied notion of convergence of group actions in this setting. For a pointed proper metric space $(X,p)$, we define the  distance between two functions $h_1,h_2 : X \to X$ as
\begin{equation}\label{d0}
d_0 (h_1, h_2) : = \inf_{r > 0 } \left\{ \frac{1}{r} + \sup_{x \in B_p(r, X)} d(h_1x, h_2x)  \right\}  .         
\end{equation} 
When we restrict this metric to $Iso(X)$ we get a left invariant (not necessarily geodesic) metric that induces the compact open topology and makes $Iso (X)$ a proper metric space. However, this distance is defined on the full class of functions $X \to X$, where it is not left invariant nor proper anymore.

Recall that if a sequence of pointed proper metric spaces $(X_i, p_i)$ converges in the pGH sense to the pointed proper metric space $(X,p)$,  one has functions $\phi_i : X_i \to X \cup \{ \ast \}$ and $\psi_i : X \to X_i$  with $\phi_i(p_i) \to p$ and satisfying Equations  \ref{pgh2}, \ref{pgh3}, \ref{pgh4}, and \ref{ghinverse}.

\begin{defn}\label{def:equivariant}
\rm Consider a sequence of pointed proper metric spaces $(X_i,p_i)$ that converges in the pGH sense to a pointed proper metric space $(X,p)$ and a sequence of closed groups of isometries $\Gamma_i \leq Iso(X_i)$. We say that the sequence $\Gamma_i$ \textit{converges equivariantly} to a closed group $\Gamma \leq Iso (X)$ if:
\begin{itemize}
\item For each $g \in \Gamma$, there is a sequence $g_i \in \Gamma_i $ with $d_0 (\psi_i g_i \phi_i , g) \to 0$ as $i \to \infty$.
\item For a sequence $g_i \in \Gamma_i$ and $g \in Iso(X)$, if there is a subsequence $g_{i_k}$ with $d_0 (\psi_{i_k} g_{i_k} \phi_{i_k} , g) \to 0$ as $k \to \infty$, then $g \in \Gamma$.
\end{itemize}
We say that a sequence of isometries $g_i \in \Gamma _i$ \textit{converges} to an isometry $g \in \Gamma$ if $d_0(\psi_i g_i \phi_i , g  ) \to 0$ as $i \to \infty$.
\end{defn}

This definition of equivariant convergence allows one to take limits before or after taking quotients \cite{fukaya-yamaguchi}.

\begin{lem}\label{equivariant}
 Let $(Y_i,q_i)$ be a sequence of proper metric spaces that converges in the pGH sense to a proper space $(Y,q)$, and  $\Gamma_i \leq Iso(Y_i)$ a sequence of closed groups of isometries that converges  equivariantly to a closed group $\Gamma \leq Iso (Y)$. Then the sequence $(Y_i/\Gamma_i, [q_i])$ converges in the pGH sense to $(Y/\Gamma , [q])$.
\end{lem}

Since the isometry groups of proper metric spaces are locally compact, one has an Arzel\'a-Ascoli type result (\cite{fukaya-yamaguchi}, Proposition 3.6).

\begin{thm}\label{equivariant-compactness}
 (Fukaya--Yamaguchi) Let $(Y_i,q_i) $ be a sequence of proper metric spaces that converges in the pGH sense to a proper space $(Y,q)$, and take a sequence $\Gamma_i \leq Iso(Y_i)$ of closed groups of isometries. Then there is a subsequence $(Y_{i_k}, q_{i_k}, \Gamma_{i_k})_{k \in \mathbb{N}}$ such that $\Gamma_{i_k}$ converges equivariantly  to a closed group $\Gamma \leq Iso(Y)$.
\end{thm}

As a consequence of Theorem \ref{gigli}, it is easy to understand the situation when the quotients of a sequence converge to $\mathbb{R}^m$.

\begin{prop}\label{gigli-corollary}
  For each $i \in \mathbb{N}$, let $(X_i, d_i, \mathfrak{m}_i, p_i)  $ be a  normalized pointed $RCD^{\ast}(- \varepsilon_i,N)$ space of rectifiable dimension $n$ with $\varepsilon_i \to 0$.  Assume $(X_i,d_i,\mathfrak{m}_i,p_i)$ converges in the pmGH sense to $(X,d^X,\mathfrak{m}, p)$, there is a sequence of closed groups of isometries $\Gamma_i \leq Iso (X_i)$ that converges equivariantly to $\Gamma \leq Iso(X)$, and the sequence of pointed proper metric spaces $(X_i/\Gamma_i , [p_i])$ converges in the pGH sense to $(\mathbb{R}^m,0)$. Then there is $c>0$ and a metric measure space $(Y, d^Y, \nu )$ such that $(X,d^X,c \mathfrak{m})$ is isomorphic to the product $(\mathbb{R}^m\times Y, d ^{\mathbb{R}^m} \times d^Y, \mathcal{H}^m \otimes \nu)$. Moreover, if $N-m \in [0,1)$,  then $Y$ is a point, and in general, $(Y,d^Y,\nu)$ is an $RCD^{\ast}(0,N-m)$ space of rectifiable dimension $\leq n-m$. Furthermore, the $\Gamma$-orbits coincide with the $Y$-fibers of the splitting $\mathbb{R}^m \times Y$.
\end{prop}

\begin{proof}
One can use the submetry $\phi: X \to  X/\Gamma  = \mathbb{R}^m$ to lift the lines of $\mathbb{R}^m$ to lines in $X$ passing through $p$. By Theorem \ref{gigli} and Remark \ref{dim-semicont-2},  we get the desired splitting $X = \mathbb{R}^m \times Y$ with the property that for some $q \in Y$, we have $\phi ( x , q)=x $ for all $x \in \mathbb{R}^m$.

 Now we show that the action of $\Gamma $ respects  the $Y$-fibers. Let $g \in \Gamma$ and assume $g(x_1,q ) = (x_2,y)$ for some $x_1,x_2 \in \mathbb{R}^m$, $y \in Y$. Then for all $t \geq 1$, one has
\begin{eqnarray*}
t \vert x_1-x_2 \vert & = & \vert \phi ( x_1+ t(x_2-x_1),q)  - \phi ( x_1,q) \vert \\
& = & \vert  \phi  ( x_1+ t(x_2-x_1),q)- \phi (x_2,y)  \vert        \\
& \leq & d^X (  ( x_1 + t (x_2-x_1),q) , (x_2,y) )\\
& = & \sqrt{    \vert  (t-1) (x_2-x_1) \vert ^2  + d^Y(q,y)^2 }.
\end{eqnarray*} 
As $t \to \infty$, this is only possible if $x_1=x_2$.
\end{proof}

In \cite{gromov-af}, Gromov studied which is the structure of discrete groups that act almost transitively on spaces that look like $\mathbb{R}^n$. Using the Malcev embedding theorem, he showed that they look essentially like lattices in nilpotent Lie groups. In \cite{breuillard-green-tao}, Breuillard--Green--Tao studied in general what is the structure of discrete groups that have a large portion acting on a space of controlled doubling. It turns out that the answer is still essentially just lattices in nilpotent Lie groups. In (\cite{zamora-fglahs}, Sections 7-9) the ideas from \cite{gromov-af} and \cite{breuillard-green-tao} are used to obtain the following structure result.

\begin{thm}\label{zamora-1}
 Let $(Z,p)$ be a proper pointed geodesic space of topological dimension $\ell \in \mathbb{N}$ and let $(D_i , p_i )$ be a sequence of pointed discrete metric spaces converging in the pGH sense to $(Z,p)$. Assume there is a sequence of isometry groups $\Gamma_i \leq Iso (D_i)$  acting transitively and with the property that for each $i$, $\Gamma_i$ is generated by its elements that move $p_i$ at most $1$. Then for large enough $i$, there are finite index subgroups $G_i \leq \Gamma_i$ and finite normal subgroups $F_i \triangleleft G_i  $ such that $G_i / F_i $ is isomorphic to a quotient of a lattice in a nilpotent Lie group of dimension $\ell$. In particular, if the groups $\Gamma_i$ are abelian, for large enough $i$ their rank is at most $\ell$. 
\end{thm}

\begin{thm}\label{zamora-2}
 Let $(X_i,d_i,\mathfrak{m}_i, p_i)$ be a sequence of simply connected pointed $RCD^{\ast}( -\varepsilon_i ,N)$ spaces with $\varepsilon_i \to 0$ admitting a sequence of discrete groups of isometries $\Gamma_i \leq Iso (X_i)$ with diam$(X_i/\Gamma _i ) \to 0$.  If the sequence $(X_i,d_i,\mathfrak{m}_i, p_i)$ converges in the pGH sense to a space $(X,d,\mathfrak{m}, p)$, then $X$ is isometric to $\mathbb{R}^m$ for some $m \leq N$.
\end{thm}

\begin{proof}
 By (\cite{zamora-fglahs}, Theorem 6), $X$ is a simply connected nilpotent Lie group. By Theorem \ref{compactness-2}, $X$ is an $RCD^{\ast}(0,N)$ space, so its metric is Riemannian. By (\cite{milnor}, Theorem 2.4), $X$ is then abelian hence isometric to $\mathbb{R}^m$ for some $m \leq N$.   
\end{proof}

\subsection{Group norms}

Let $(X,p)$ be a pointed proper geodesic space and $\Gamma \leq Iso(X)$ a group of isometries. The \textit{norm} $\Vert \cdot \Vert_p : \Gamma \to \mathbb{R}$ associated to $p$ is defined as $\Vert g \Vert_p : = d(gp,p)$.  We denote as $\mathcal{G}(\Gamma, X ,  p,r)$ the subgroup of $\Gamma$ generated by the elements of norm $\Vert \cdot \Vert_p \leq r$. The \textit{spectrum} $\sigma (\Gamma )$ is defined as the set of $r \geq 0 $ such that  $\mathcal{G}(\Gamma ,X, p,r) \neq \mathcal{G}(\Gamma ,X, p , r- \varepsilon )$ for all $\varepsilon > 0 $.  Notice that we always have $0 \in \sigma (\Gamma)$.  If we want redundancy we sometimes write $\sigma(\Gamma , X,p)$ to denote the spectrum of the action of $\Gamma$ on the pointed space $(X,p)$. This spectrum is closely related to the covering spectrum introduced by Sormani--Wei in \cite{sormani-wei-spectrum}, and it also satisfies a continuity property.   

\begin{prop}\label{spec-cont}
 Let $(X_i,p_i)$ be a sequence of pointed proper metric spaces that converges in the pGH sense to $(X,p)$ and consider a sequence of closed isometry groups $\Gamma _i \leq Iso (X_i)$ that converges equivariantly to a closed group $\Gamma \leq Iso (X)$. Then for any convergent sequence $r_i \in \sigma (\Gamma _ i ) $, we have $\lim_{i \to \infty} r_i \in \sigma (\Gamma ) $.
\end{prop}

\begin{proof}
Let $r = \lim_{i \to \infty}r_i$ By definition, there is a sequence $g_i \in \Gamma_i$ with $\Vert g_i \Vert_{p_i} = r_i$, and $g_i \notin \mathcal{G}(\Gamma _i , X_i, p_i, r_i - \varepsilon )$ for all $\varepsilon >0$. Up to subsequence, we can assume that $g_i$ converges to some $g \in Iso (X)$ with $\Vert g \Vert_p = r$.  

If $r \notin \sigma (\Gamma )$, it would mean there are $h_1, \ldots , h_k \in \Gamma $ with $\Vert h_j \Vert_p  < r$ for each $j \in \{1, \ldots , k \}$, and $h_1 \cdots h_k = g$. For each $j$, choose sequences $h_j^i \in \Gamma _i$ that converge to $h_j$. As the norm is continuous with respect to convergence of isometries, for $i$ large enough one has $ \Vert h_j^i \Vert_{p_i} < r_i $ for each $j$. 

The sequence $g_i ( h_1^i \cdots h_k^i)^{-1}  \in \Gamma_i$ converges to $g(h_1 \cdots h_k )^{-1}=e \in \Gamma$, so its norm is less than $r_i$ for $i$ large enough, allowing us to write $g_i$ as a product of $k+1$ elements with norm $< r_i$, thus a contradiction.
\end{proof}

\begin{rem}
\rm It is possible that an element in $\sigma (\Gamma)$ is not a limit of elements in $\sigma (\Gamma_i)$, so this spectrum is not necessarily continuous with respect to equivariant convergence (see \cite{kapovitch-wilking}, Example 1).
\end{rem}

Let $(X,p)$ be a pointed proper metric space and $\Gamma \leq Iso(X)$ a group of isometries. A \textit{short basis} of $\Gamma$ with respect to $p$ is a countable collection $\{    \gamma _1 , \gamma_2 , \ldots \} \subset \Gamma $ such that $\gamma_{j+1} $ is an element of minimal norm $\Vert \cdot \Vert_p$ in $\Gamma \backslash \langle \gamma_1 , \ldots , \gamma_j \rangle$.   Notice that if $\beta \subset \Gamma $ is a short basis with respect to $p$, then
\[  \sigma ( \Gamma , X, p ) = \{ 0 \} \cup \{ \Vert \gamma \Vert _p \vert \gamma \in \beta  \}      .   \]
If $\Gamma$ is discrete, a short basis always exists, and if additionally $\Gamma$ is finitely generated, then any short basis is finite. If a doubling condition is assumed, one can control how many elements of the spectrum lie on a compact subinterval of $(0,\infty )$.

\begin{prop}\label{long-short-basis}
  Let $(X,d,\mathfrak{m}, p ) $ be a pointed metric measure space, and assume that for each $R > r > 0$, there is $C(R,r) > 0 $ with the property that for each $x \in X$ one has 
\[\mathfrak{m}(B_x(R,X)) \leq C \cdot \mathfrak{m}(B_x(r,X)) . \]
Then for each $[a, b ] \subset (0, \infty )$, any discrete group $\Gamma \leq Iso (X)$, and any short basis $\beta \subset \Gamma$, there are at most $C(3 b  , a /2 )$ elements of $\beta$ with norm in $[a,b]$.
\end{prop}

\begin{proof}
Notice that if $g,h \in \Gamma$ are distinct elements of a short basis with $\Vert g \Vert_p , \Vert h \Vert_p \geq a$, one must have $d(gp,hp) \geq a$. Otherwise, $d(g^{-1}h p , p) < a $ and one must take $g^{-1}h$ as an element of the short basis instead of $g$ or $h$.
 
Now assume one has a discrete group of isometries $\Gamma \leq Iso(X)$ having a short basis $\beta$ with $N$ elements $\{ g_1 , \ldots , g_N \}$ whose norm is in $[a,b]$. Since the balls of the collection $ \{  B_{g_jp}(a/2,X) \}_{j=1}^N$ are disjoint and all of them are contained in $B_{g_{\ell}p}(3b,X)$ for each $\ell \in \{1 , \ldots , N \}$, we get
\[ \sum_{j=1}^N \mathfrak{m}(B_{g_jp}(a/2,X)) \leq \mathfrak{m}(B_{g_{\ell}p}(3b)) \leq C (3b , a/2)\cdot \mathfrak{m}(  B_{g_{\ell}p}(a/2,X)).  \] 
Summing over all $\ell$, we conclude that $N \leq C(3b,a/2)$.
\end{proof}

It is well known that around a regular point, the spectrum displays a gap phenomenon.

\begin{lem}\label{gap}
 (Gap Lemma) Let $(Y_i,d_i^Y,\mathfrak{m}_i^Y, p_i)$ be a sequence of pointed $RCD^{\ast}(K,N)$ spaces and  $\Gamma_i \leq Iso (Y_i) $ a sequence of discrete groups of measure preserving isometries such that the sequence of $RCD^{\ast}(K,N)$ spaces $(Y_i/\Gamma_i, [p_i])$ converges in the pmGH sense to an $RCD^{\ast}(K,N)$ space $(X, q)$ with $q$ an $m$-regular point. Then there is $\eta > 0 $ and $\eta_i \to 0$ such that for $i$ large enough we have
\[  \sigma (\Gamma _i , Y_i , p_i) \cap [\eta_i , \eta] = \emptyset .   \]
\end{lem}

\begin{proof}
Since $q$ is $m$-regular, we can construct $\eta_i > 0 $ converging to $0$ so slowly that for any $\lambda_i \to \infty$ with $\lambda_i  \eta_i ^2 \to 0$, the sequence $(\lambda_i Y_i / \Gamma _i , [p_i])$ converges in the pGH sense to $(\mathbb{R}^m , 0)$. The result then follows from the following claim.
\begin{center}
\textbf{Claim: }For each  $r_i \in \sigma (\Gamma_i, Y_i , p_i)$ with $r_i \geq \eta_i $, we have $\liminf_{i\to \infty} r_i > 0$.
\end{center}

If after taking a subsequence, $r_i \to 0$, then by our choice of $\eta_i$, the sequence $\left(   \frac{1}{r_i } Y_i / \Gamma _i , [p_i ]   \right)$ converges in the pGH sense to $(\mathbb{R}^m , 0)$. Then by Proposition \ref{gigli-corollary}, after taking again a subsequence we can assume that $\left( \frac{1}{r_i} Y_i,p_i \right)$ converges in the pGH sense to $(\mathbb{R}^m \times Z, (0,z))$ for some geodesic space $Z$, and $\Gamma_i $ converges equivariantly to a closed group $\Gamma \leq Iso (\mathbb{R}^m \times Z) $ in such a way that the $\Gamma$-orbits coincide with the $Z$-fibers. This implies that $\sigma (\Gamma ) = \{ 0 \}$ (see Lemma \ref{gen} and Remark \ref{continuous-spectrum} below), but $1 \in \sigma \left( \Gamma_i , \frac{1}{r_i} Y_i , p_i  \right) $ for each $i$, contradicting Proposition \ref{spec-cont}.
\end{proof}

\subsection{Discrete approximations of Lie groups}

This elementary theorem of Jordan is one of the first results in geometric group theory (\cite{thurston}, Section 4.1). We denote the group of unitary operators on a Hilbert space $H$ by $U(H)$, and if $H = \mathbb{C}^m$, we denote $U(H)$ by $U(m)$.

\begin{thm} \label{jordan}
 (Jordan) For each $m \in \mathbb{N}$, there is $C > 0$ such that any finite subgroup $\Gamma \leq U (m)$ contains an abelian subgroup $A \leq \Gamma $ of index $[\Gamma : A] \leq C$.
\end{thm}

By Peter--Weyl Theorem, Theorem \ref{jordan} applies to arbitrary compact Lie groups \cite{zimmer}.

\begin{thm}\label{pw}
 (Peter--Weyl) If $G$ is a compact Lie group, then it has a faithful unitary representation $\rho :G \to U(H) $ for some finite dimensional Hilbert space $H$.
\end{thm}

A torus $\mathbb{R}^m / \mathbb{Z}^m$ has arbitrarily dense finite subgroups. The following result generalizes this fact to non-connected Lie groups (\cite{gelander}, Lemma 3.5).

\begin{thm}\label{Gelander}
 (Gelander, Kazhdan, Margulis, Zassenhaus) Let $G$ be a compact Lie group, whose identity component is abelian. Then there is a dense subgroup which is a direct limit of finite subgroups.
\end{thm}

\begin{cor}\label{dense-abelian}
 Let $G$ be a compact Lie group. Then there is $C>0$ such that any virtually abelian subgroup $H \leq G$ has an abelian subgroup $A\leq H$ of index $[H, A] \leq C$.
\end{cor}

\begin{proof}
Since $H$ is virtually abelian, then so is its closure $\overline{H} \leq G$. This means that the identity connected component of $\overline{H}$ is abelian. By Theorem \ref{Gelander}, there is a dense subgroup $H^{\prime}\leq \overline{H}$ that is a direct limit of finite subgroups. By Theorems \ref{jordan} and \ref{pw}, there is $C>0$, depending only on $G$, such that all those finite groups have abelian subgroups of index $\leq C$, and hence the same holds for $H^{\prime}$ and $\overline{H^{\prime}} = \overline{H}$. Since this property passes to subgroups, $H$ satisfies it too.
\end{proof}

One of the most important applications of Theorem \ref{jordan} is the structure of crystallographic groups \cite{charlap}.

\begin{thm}\label{bieber}
 (Bieberbach) There is $C(m)>0$ such that any discrete group of isometries $\Gamma \leq Iso (\mathbb{R}^m)$ has an abelian subgroup $A \leq \Gamma$ of index $[\Gamma : A ] \leq C$.
\end{thm}

For amenable groups, approximate representations are close to actual representations (\cite{kazhdan}, Theorem 1). In the statement of Theorem \ref{kazhdan}, $\Vert \cdot \Vert$ denotes the operator norm.

\begin{thm} \label{kazhdan}
 (Kazhdan) Let $\Gamma$ be an amenable group, $H$ a Hilbert space, $\varepsilon \in [0, 1/100] $, and $\rho : \Gamma \to U(H) $ a map such that for all $g,h \in \Gamma$, one has $\Vert \rho (g) \rho (h) - \rho (gh) \Vert \leq \varepsilon $. Then there is a representation $\rho_0 : \Gamma \to U(H)$ such that for all $g \in \Gamma $ one has $ \Vert \rho (g) - \rho_0(g) \Vert \leq \varepsilon  $.
\end{thm}

The following notion of discrete approximation is closely related to the concepts of \textit{approximable groups} studied by Turing in \cite{turing} and \textit{good model} studied by Hrushovski in \cite{hrushovski} and by Breuillard--Green--Tao in \cite{breuillard-green-tao}. 

\begin{defn}\label{def:discreteapprox}
\rm Let $\Gamma$ be a Lie group with a left invariant metric and $\Gamma_i$ a sequence of discrete groups. We say that a sequence of functions $f_i : \Gamma _ i \to \Gamma$ is a \textit{discrete approximation} of $\Gamma$ if there is $R_0 > 0$ such that  $f_i^{-1}(B_e(R_0,\Gamma )) \subset \Gamma_i $ is a generating set containing the identity for large enough $i$, and for each $R>0$, $\varepsilon > 0$ there is $i_0(R, \varepsilon ) \in \mathbb{N}$ such that for all $i \geq i_0$ the following holds:
\begin{itemize}
\item the preimage $f_i^{-1}(B_e(R,\Gamma ))$ is finite.
\item $f_i(\Gamma_i)$ intersects each ball of radius $\varepsilon $ in $B_e(R,\Gamma)$.
\item For each $g_1,g_2 \in f_i^{-1}(B_e(R, \Gamma))$, one has
\[  d(f_i(g_1g_2^{-1}), f_i(g_1)f_i(g_2^{-1})) \leq \varepsilon .        \]
\end{itemize}
\end{defn}
An immediate consequence of the above definition is that if a Lie group $\Gamma$ with a left invariant metric $d^{\Gamma}$ admits a discrete approximation $f_i : \Gamma_i \to \Gamma$ then $(\Gamma , d^{\Gamma })$ is a proper metric space.
\begin{lem}\label{cci}
 Let $\Gamma$ be a Lie group with a left invariant metric, $\Gamma_0 \triangleleft $ $ \Gamma$ the connected component of the identity, $\Gamma_i$ a sequence of discrete groups, and $f_i : \Gamma_i \to \Gamma$ a  discrete approximation. Let $r > 0 $ be such that $B_e(r,\Gamma ) \subset \Gamma_0$. For each $i \in \mathbb{N}$ let $G_i \leq \Gamma_i$ denote the subgroup generated by $f_i^{-1}(B_e(r,\Gamma ))$. Then for $i$ large enough, $G_i$ is a normal subgroup of $\Gamma_i$.
\end{lem}
\begin{proof}
 First we show that for any fixed $\delta \in (0, r]$, the subgroup of $\Gamma_i$ generated by $f_i^{-1}(B_e(\delta , \Gamma ))$ in $\Gamma_i$ coincides with $G_i$ for large enough $i$. To see this, first take a collection $y_1, \ldots , y_{n} \in B_e(r,\Gamma )$ with 
\[ B_e(r, \Gamma ) \subset \bigcup_{j=1}^n B_{y_j}(\delta / 10 , \Gamma) .  \]
By connectedness,  for each $j  \in \{1 , \ldots , n \} $ we can construct a sequence $e= z_{j,0}, \ldots , z_{j,k_j}=y_j $ in $\Gamma$ with  $d(z_{j,\ell -1},z_{j , \ell}) \leq \delta /10$ for each $\ell \in \{ 1, \ldots ,k_j \} $. If $R > 0 $ is such that  each $  z_{j, \ell}  $ lies in $B_e(R/2, \Gamma )$, then for each $i \geq i_0(R, \delta /10 ) $ and any any element $x \in f_i^{-1}(B_e(r, \Gamma ))$ we can find $y_j $ with $d(y_j,f_i(x)) \leq \delta /10$, and $ e = x_0 , \ldots , x_{k_j} = x $  in $ \Gamma _i $ with $d(z_{j,\ell} , f_i(x_{\ell})) \leq \delta / 10$ for each $\ell$. This allows us to write $x = (x_1)(x_1^{-1}x_2 )\cdots (x_{k_j -1}^{-1}x_{k_j})$ as a product of $k_j $ elements in $f_i^{-1}(B_e(\delta , \Gamma ))$, proving our claim.

Let $R_0 > 0 $ be given by the definition of discrete approximation, and choose $\delta > 0 $ small enough such that for all $x \in B_e(2R_0,\Gamma ) $, $y \in B_e(\delta , \Gamma )$ one has $xyx^{-1} \in B_e(r/2, \Gamma)$.  Then for large enough $i$, the conjugate of an element in  $f_i^{-1}(B_e(\delta , \Gamma ))$ by an element in $f_i^{-1}(B_e(R_0, \Gamma ))$ lies in $f_i^{-1}(B_e(r,\Gamma ))$. Since $f_i^{-1}(B_e(R_0, \Gamma ))$ generates $\Gamma_i$, this shows that $G_i$ is normal in $\Gamma_i$.
\end{proof}

\begin{defn}
\rm Let $G$ be a group and $S \subset G$ a generating subset containing the identity. We say that $S$ is a \textit{determining set} if $G$ has a presentation $G  = \langle S \vert R \rangle $ with $R$ consisting of words of length $ 3 $ using as letters the elements of $S \cup S^{-1}$.
\end{defn}

\begin{exa}
\rm The subset $ \{ [-k] , [1-k], \ldots , [k-1], [k] \} \subset \mathbb{Z}/n\mathbb{Z}$ is a determining set if and only if $k \geq n/3$.
\end{exa}

\begin{exa}
\rm If $X$ is a compact Riemannian manifold and $p \in X$, then 
\[  \{ [ \gamma ] \in \pi_1(X, p) \vert \text{ length}(\gamma ) \leq 2 \cdot \text{diam}(X)  \}  \]
is a determining set of $\pi_1(X,p)$ (see \cite{gromov-ms}, Proposition 5.28).
\end{exa}

\begin{rem}\label{s-and-sm}
\rm Let  $G$ be a group and $S \subset G$ a determining set. It is easy to see that $S^M$ is also a determining set for all $M \geq 1$.
\end{rem}

\begin{defn}
\rm A discrete approximation is said to be \textit{clean} if there is an open generating set  $S_0 \subset \Gamma $, a compact subset $K_0 \subset \Gamma$, and for large enough $i$ symmetric determining sets $S_i \subset \Gamma_i$ with
\[  f_i^{-1}(S_0) \subset S_i \subset f_i^{-1}(K_0) .   \] 
\end{defn}

\begin{exa}
\rm If $\Gamma$ is compact, any discrete approximation is clean with $S_0 = K_0 = \Gamma$. 
\end{exa}

\begin{exa}\label{no-clean}
\rm Take $\Gamma_i = \mathbb{Z} / 3i \mathbb{Z}$, $\Gamma = \mathbb{Z}$, equip $\Gamma$ with any left invariant proper metric,  and define $f_i : \Gamma_i \to \Gamma$ by 
\[   
f_i ([n]) =
     \begin{cases}
       n &\quad\text{if }n \in \{ -i, \ldots , i \} \\
       i &\quad\text{if }  n \in \{ i+1, \ldots 2i - 1 \} \\ 
     \end{cases}
\]
Then the sequence $f_i$ is a discrete approximation but it is not clean: for any compact $K_0 \subset \Gamma$, there is $n \in \mathbb{Z}$ with $ K_0 \subset \{ -n, \ldots , n\}$, but for $i > n$, the group $\Gamma_i$ does not contain any determining set contained in $ \{  [-n], \ldots , [n]  \} $. 
\end{exa}

The following proposition is a simple exercise in real analysis.

\begin{defn}\label{small-1}
\rm Let $\Gamma$ be a Lie group with a left invariant metric, $\Gamma_i$ a sequence of discrete groups, and $f_i : \Gamma_i \to \Gamma$ a discrete approximation. We say a sequence of subgroups $H_i \leq \Gamma_i $ consists of \textit{small subgroups} if for any $h_i \in H_i$, we have $f_i (h_i) \to e$.
\end{defn}

\begin{prop}\label{small-quotient}
 Let $\Gamma$ be a Lie group with a left invariant metric, $\Gamma_i$ a sequence of discrete groups, $f_i : \Gamma_i \to \Gamma$ a discrete approximation, $H_i \triangleleft \Gamma_i$ a sequence of small normal subgroups, and $\theta_i : \Gamma _i / H_i \to \Gamma_i$ a sequence of sections. Then the functions $\tilde{f}_i : \Gamma _i / H_i \to \Gamma $ defined as $\tilde{f}_i = f_i \circ \theta_i$ form a discrete approximation. Moreover, the sequence  $\tilde{f}_i$ is clean if and only if the sequence $f_i$ is clean.
\end{prop}

Turing determined which compact Lie groups were approximable by finite groups \cite{turing}.

\begin{thm}\label{turing}
 (Turing) Let $\Gamma$ be a compact Lie group with a left invariant metric, $\Gamma_i$ a sequence of discrete groups, and $f_i : \Gamma_i \to \Gamma$ a discrete approximation. Then there is $C>0$ and a sequence of small normal subgroups $H_i \triangleleft \Gamma_i$ such that $\Gamma_i / H_i$ has an abelian subgroup $A_i$ of index $[\Gamma_i / H_i : A_i] \leq C$ for all $i$.  In particular, the connected component of the identity in $\Gamma$ is abelian.
\end{thm}

A non-compact and considerably deeper version of Turing Theorem was obtained by Breuillard--Green--Tao \cite{breuillard-green-tao}.

\begin{defn}
\rm Let $G$ be a group,  $u_1, \ldots , u_r   \in G $, $N_1, \ldots , N_r \in \mathbb{R}^{+}$, and $C> 0 $. The set $P = P(u_1, \ldots , u_r; N_1, \ldots , N_r )\subset G$ is defined as the set of elements of $G$ that can be written as words in the $u_j$'s and their inverses such that the number of appearances of $u_j$ and $u_j^{-1}$ is not more than $N_j$ for each $j$. We say that  $P$ is a \textit{nilprogression} of rank $r$ in $C$-regular form if 
\begin{itemize}
\item For each $1\leq j\leq k \leq r$ and any choice of signs,
\[[u_j^{\pm 1}, u_{k}^{\pm 1}] \in P\left( u_{k+1}, \ldots , u_r ;\frac{ CN_{k+1}}{N_jN_k} , \ldots , \frac{CN_r}{N_jN_k} \right) \]  
\item The expressions $u_1^{n_1} \ldots u_r^{n_r}$ represent distinct elements as  $n_1, \ldots , n_r$ range over the integers with $\vert n_j \vert  \leq  N_j/C $ for each $j$.
\end{itemize}
  The minimum of $N_1, \ldots , N_r$ is called the \textit{thickness} of $P$ and is denoted as thick$(P)$. The set $G(P) :=  \{ u_1^{n_1} \ldots u_r^{n_r} \vert \text{ } \vert n_j  \vert \leq N_j/C \} $ is called the \textit{grid part} of $P$.  For $\varepsilon \in (0,1]$, the set $P(u_1, \ldots , u_r; \varepsilon N_1, \ldots , \varepsilon N_r )$ is also a nilprogression in $C$-regular form and will be denoted as $\varepsilon P$.
\end{defn}

\begin{thm}\label{bgt}
 (Breuillard--Green--Tao) Let $\Gamma$ be an $r$-dimensional Lie group with a left invariant metric, $\Gamma_i$ a sequence of discrete groups, and $f_i : \Gamma_i \to \Gamma$ a discrete approximation. Then there is $C>0$, a sequence of small normal subgroups $H_i \triangleleft \Gamma_i$, and rank $r$  nilprogressions $P_i \subset \Gamma_i / H_i$ in $C$-regular form with thick$(P_i) \to \infty$ and satisfying that for each $\varepsilon \in (0,1]$ there is $\delta > 0 $ with  
\begin{itemize}
    \item $\tilde{f}_i^{-1}(B_e(\delta , \Gamma )) \subset G(\varepsilon P_i )$ for $i$ large enough,
    \item $ \tilde{f}_i (G(\delta P_i ))\subset B_e(\varepsilon , \Gamma)  $ for $i$ large enough, 
\end{itemize}
where $\tilde{f}_i : \Gamma _ i / H_i \to \Gamma$ is the discrete approximation given by Proposition \ref{small-quotient}. In particular, the connected component of the identity in $\Gamma$ is nilpotent.
\end{thm}

\begin{proof} 
 This is the main result of \cite{breuillard-green-tao}. The condition about the grid part is obtained in (\cite{zamora-fglahs}, Theorem 106) by a careful inspection of the original proof.
\end{proof}

Lattices in simply connected nilpotent Lie groups have large portions that look like nilprogressions. The converse is known as the Malcev embedding theorem (\cite{breuillard-green-tao}, Lemma C.3).

\begin{thm}\label{malcev}
 (Malcev) There is $N (C, r) > 0$ such that for any nilprogression $P(u_1, \ldots , u_r;$ $ N_1, \ldots , N_r )$ in $C$-regular form with thickness $\geq N$, there is a unique polynomial nilpotent group structure $\ast _P : \mathbb{R}^r \times \mathbb{R}^r \to \mathbb{R}^r$ of degree $\leq N$ that restricts to a group structure on the lattice $\mathbb{Z} ^r$ and such that the map $(n_1 , \ldots , n_r) \to u_1^{n_1} \ldots , u_r^{n_r}$ is a group morphism from $(\mathbb{Z}^r, \ast_P )$ onto the group generated by $P$, injective in the box $\mathbb{Z}^r \cap \left( \left[ -\frac{N_1}{C}, \frac{N_1}{C} \right] \times \cdots \times \left[ -\frac{N_r}{C}, \frac{N_r}{C} \right] \right)$, which is a determining set of $(\mathbb{Z}^r, \ast _P)$.
\end{thm}

\subsection{Equivariant convergence and discrete approximations}\label{reformulations}

Let $(X_i,p_i)$ be a sequence of pointed proper geodesic spaces that converges in the pGH sense to a space $(X,p)$ whose isometry group $Iso (X)$ is a Lie group. Assume there is a sequence of discrete groups $\Gamma_i \leq Iso (X_i) $ that converges equivariantly to a closed group $\Gamma \leq Iso (X)$ and such that diam$(X_i / \Gamma _i ) \leq C$ for some $C>0$. In this section we discuss how to build a discrete approximation $f_i : \Gamma _i \to \Gamma $ out of this situation. 

Recall that from the definition of pGH convergence one has functions $\phi_i : X_i \to X \cup \{ \ast \}$ and $\psi_i :X \to X_i$  with $\phi_i(p_i) \to p$ and satisfying Equations  \ref{pgh2}, \ref{pgh3}, \ref{pgh4}, and \ref{ghinverse}.

With these functions, one could define $f_i^{\prime} : \Gamma_ i \to \Gamma \cup \{ \ast \}$ in a straightforward way: for each $g \in \Gamma_i$, if there is an element $\gamma \in \Gamma$ with $d_0(\phi_i \circ g \circ \psi_i, \gamma)\leq 1$, choose $f_i^{\prime} (g)$ to be an element of $\Gamma$ that minimizes $d_0( \phi_i \circ g \circ \psi_i , f_i^{\prime} (g) )$. Otherwise, set $f_i^{\prime} (g) = \ast$. Then we divide the situation in two cases to obtain $f_i: \Gamma _i \to \Gamma$ out of $f_i^{\prime}$. In either case, it is straightforward to verify that $f_i$ is a discrete  approximation.

\begin{center}
\textbf{Case 1:} $X$ is compact.
\end{center}
The sequence diam$(X_i)$ is bounded, and there is no need to use $\ast$ to define $f_i^{\prime}$ for $i$ large enough. Then the functions $f_i = f_i^{\prime} : \Gamma_i \to \Gamma$ as defined above will be a discrete approximation. Notice that in this case, $\Gamma$ is compact and $\Gamma_i$ is finite for large enough $i$.

\begin{center}
\textbf{Case 2:} $X$ is non-compact.
\end{center}

The group $\Gamma$ is non-compact and we can find a sequence $g_i \in \Gamma $ with $d(e,g_i) \to \infty$. Then define $f_i : \Gamma_i \to \Gamma$ as
\begin{equation*}
f_i (g):=
\begin{cases}
f_i^{\prime} (g) & \text{if }\, f_i^{\prime}(g) \in \Gamma , \\
g_i & \text{if }\, f_i(g) = \ast .
\end{cases}
\end{equation*}
Notice that we need the fact that the sequence diam$(X_i / \Gamma_i)$ is bounded to guarantee that $\Gamma $ is non-compact.

\begin{rem}\label{continuous-approximation}
\rm In case the groups $\Gamma_i$ were not discrete, the construction of the functions $f_i: \Gamma_i \to \Gamma$ described above still works, but the finiteness condition in the defintion of discrete approximation (and hence Theorems \ref{turing} and \ref{bgt}) may fail to hold.
\end{rem}

Due to our construction, the maps $f_i : \Gamma_i \to \Gamma$ are actually Gromov--Hausdorff approximations when we equip the groups with the metric $d_0$ from Equation \ref{d0}. In particular, we have the following.

\begin{prop}\label{discrete-gh}
 For each $R > 0$, and $\varepsilon > 0 $, there is $i_0 (R, \varepsilon) \in \mathbb{N}$ such that for $i \geq i_0$ we have
\begin{itemize}
\item $f_i (B_e(R,\Gamma_i )) \subset B_e(R+\varepsilon , \Gamma) $.
\item $f_i^{-1}(B_e(R, \Gamma )) \subset B_e(R + \varepsilon , \Gamma_i)$.
\end{itemize}
\end{prop}

\begin{cor}\label{small-characterization}
 A sequence of subgroups $H_i \leq \Gamma_i$ consists of small subgroups in the sense of Definition \ref{small-1} if and only if $d_0 (h_i,Id_{X_i}) \to 0$ for any choice of $h_i \in H_i$.
\end{cor}

\begin{rem}\label{small-2}
\rm Due to Corollary \ref{small-characterization}, if we have a sequence of pointed proper geodesic spaces $(X_i,p_i)$, we will say that a sequence of groups of isometries $H_i \leq Iso (X_i)$ consists of \textit{small subgroups} if $d_0 (h_i,Id_{X_i}) \to 0$ for any choice of $h_i \in H_i$.
\end{rem}

 A proof of the following well known lemma can be found in (\cite{zamora-fglahs}, Section 2.4).

\begin{lem}\label{gen}
 Let $D>0$,  $(Y,q)$ a pointed proper geodesic space, and $G \leq Iso (Y)$ a closed group of isometries with diam$(Y/G) \leq D$. Then $\{ g \in G \vert  \Vert g \Vert_q \leq 3D \}$ is a generating set of $G$.
\end{lem}

\begin{rem}\label{continuous-spectrum}
\rm Notice that in the particular case when $G $ acts transitively, Lemma \ref{gen} implies that $\sigma (G) = \{ 0 \}$. 
\end{rem}

In order to obtain clean discrete approximations, we consider universal covers.

\begin{lem}\label{det}
 Let $D>0$, $(Y,q)$ a pointed proper geodesic space and $G \leq Iso (Y)$ a closed group of isometries with diam$(Y/G) \leq D$. If $\{ g \in G \vert \Vert g \Vert_q \leq 20D \}$ is not a determining set, then there is a non-trivial covering map $\tilde{Y} \to Y$.
\end{lem}

\begin{proof}    
The result follows from (\cite{zamora-fglahs}, Theorem 79) which shows   that  the construction in (\cite{zamora-fglahs}, Section 2.12) produces, under the above hypotheses, a non-trivial covering map $\tilde{Y} \to Y$. 
\end{proof}

\begin{cor}\label{clean-uc}
 Let $(X_i,d_i,\mathfrak{m}_i)$ be a sequence of $RCD^{\ast}(K,N;D)$ spaces and $(\tilde{X}_i,\tilde{d}_i,\tilde{\mathfrak{m}}_i)$ their universal covers.  Assume for some choice of points $p_i \in \tilde{X}_i$, the sequence  $(\tilde{X}_i,p_i)$ converges in the pmGH sense to a pointed $RCD^{\ast}(K,N)$ space $(X,p)$ and the sequence $\pi_1(X_i)$  converges equivariantly to a closed group $\Gamma \leq Iso (X)$. Then the discrete approximation $f_i : \pi_1(X_i) \to \Gamma$ constructed above is clean.
\end{cor}

\begin{proof}
Let $S_i : = \{ g \in \pi_1(X_i) \vert  \Vert g \Vert_{p_i}  \leq 20 D \}     . $ Then for each sequence $g_i \in S_i$, the sequence $f_i (g_i) \in \Gamma$ is bounded, so the sets $f_i (S_i) $ are contained in a compact set $K_0 \subset \Gamma$. On the other hand, by Lemma \ref{gen} the sequence $T_i : =  \{ g \in \pi_1(X_i) \vert \Vert g \Vert_{p_i} \leq 4 D \}  $ converges to a compact generating set $T \subset \Gamma$ containing a neighborhood of the identity. Hence $S_0 : = $ int$(T^2)$ is an open generating set of $\Gamma$ and for large enough $i$ the set  $f_i^{-1}(S_0)$ is contained in $S_i$. By Lemma \ref{det}, $S_i$ is determining and the sequence $f_i$ is  clean. 
\end{proof}

\section{Bounded generation}\label{section-3}

In this section we prove Theorem \ref{fg}  following the lines of \cite{kapovitch-wilking}, but with tools  adapted to the non-smooth setting.

\begin{thm}\label{short-short}
 There is $C(K,N,R,r) > 0 $ such that the following holds.  Let $(X,d,\mathfrak{m},p)$ be a pointed $RCD^{\ast}(K,N)$ space, and $\Gamma \leq Iso (X)$ a discrete group of measure preserving isometries. Then there is a point $q \in B_p(r,X)$ with the property that any short basis of $\Gamma$ with respect to $q$ has at most $C$ elements of norm $\Vert \cdot \Vert _q \leq 3R$. 
\end{thm}

\begin{customthm}{1}
 Let $(X,d,\mathfrak{m})$ be an $RCD^{\ast}(K,N;D)$ space. Then $\pi_1(X)$ can be generated by $\leq C(N,K,D)$ elements.
\end{customthm}

\begin{proof}[Proof of Theorem \ref{fg}:]
Let $\tilde{X}$ be the universal cover of $X$. By Theorem \ref{short-short} there is $x \in \tilde{X}$ such that any short basis of $\pi_1(X)$ with respect to $x$ has at most $\leq C(K,N,D)$ elements of norm $\Vert \cdot \Vert _x \leq 3D$. On the other hand, by Lemma \ref{gen} no element of such short basis can have norm $\Vert \cdot \Vert _x > 3D$ and hence $C(K,N,D)$ elements generate $\pi_1(X)$.
\end{proof}

\begin{proof}[Proof of Theorem \ref{short-short}:]
Since for $K > 0$ any $RCD^{\ast}(K,N)$ space is automatically an $RCD^{\ast}(0,N)$ space, it is enough to assume $K \leq 0$. As in \cite{kapovitch-wilking}, we call a sequence of quadruples $(N_i,X_i,p_i,\Gamma_i)$ a \textit{contradicting sequence} if for each $i$, $N_i \in \mathbb{N}$, $(X_i,p_i)$ is a pointed $RCD^{\ast}(K,N) $ space, and  $\Gamma_i \leq Iso (X_i)$ is  a discrete group of measure preserving isometries such that
\begin{itemize}
\item For each $q \in B_{p_i}(r, X_i)$, $\Gamma_i$ has a short basis with respect to $q$ with at least $N_i$ elements of norm $\Vert \cdot \Vert_q \leq 3R$.
\item $N_i \to \infty$.
\end{itemize}

The first crucial observation is that we can blow up such a contradicting sequence to obtain another one. That is, if $\lambda_i \to \infty $ slowly enough, then for any $q \in  B_{p_i}(r,X_i)$ the short basis of the action of $\Gamma_i$ on $ X_i$ with respect to $q$ will have more than $ N_i /2$ elements of length $\Vert \cdot \Vert_q \leq 3R/\lambda_i$. This is because by Theorem \ref{bg-in} and Proposition \ref{long-short-basis}, the number of elements of such short basis with norm $\Vert \cdot \Vert _ q$ in $[3R/\lambda  _i , 3R ]$ is controlled, so we can arrange for most of the elements to have norm $\Vert \cdot \Vert _ q$ in $[0,3R / \lambda_i]$.

By Corollary \ref{quotient}, given any contradicting sequence $(N_i, X_i, p_i,\Gamma_i)$, we can always assume after taking a subsequence and renormalizing the measure (notice that the measure does not play any role in the definition of contradicting sequence so we can renormalize it at will) that the quotient $(X_i/\Gamma_i,[p_i])$ converges in the pmGH sense to a pointed $RCD^{\ast}(K,N)$ space $(X,p)$ of rectifiable dimension $m$. The proof is done by reverse induction on  $m$, which ends in finitely many steps since necessarily $m \leq N$.

 Let $p_0 \in B_p(r/2,X)$ be an $m$-regular point and $p_i^{\prime} \in B_{p_i }(3r/4, X_i)$ be such that $[p_i^{\prime}]$ converges to $p_0$. By the observation above, if $\lambda_i \to \infty$ slowly enough $(N_i/2, \lambda_i X_i,p_i^{\prime},\Gamma _i )$ is a contradicting sequence with the additional property that $(\lambda_i X_i /\Gamma_i , [p_i^{\prime}])$ converges in the pmGH sense (after renormalizing the measure around $[p_i^{\prime}]$) to $(\mathbb{R}^m, d^{\mathbb{R}^m}, \mathcal{H}^m, 0)$. By Theorem \ref{zoom}, there is a sequence $q_i \in B_{p_i^{\prime}}(r/2, \lambda_i X_i)$ such that for any sequence $\alpha_i \to \infty$, the sequence $(\alpha_i \lambda_i X_i / \Gamma_i , [q_i])$ converges in the pGH sense to a space containing an isometric copy of $\mathbb{R}^m$. Since this holds for any sequence $\alpha_i$, we choose it as follows:
 \begin{itemize}
 \item We make it diverge so slowly that $(N_i/4, \alpha_i \lambda_i X_i , q_i, \Gamma_i)$ is a contradicting sequence.
 \item  For $i$ large enough,  $1/\alpha_i \in \sigma (\Gamma _i  , \lambda_i X_i, q_i^{\prime})$ for some $q_i^{\prime} \in B_{q_i}(r, \alpha_i  \lambda_iX_i)$.
 \end{itemize}
The following claim will finish the induction step, constructing a contradicting sequence whose quotient converges to an $RCD^{\ast}(K,N)$ space of rectifiable dimension strictly greater than $m$.
\begin{center}
\textbf{Claim: }The sequence $(\alpha_i \lambda_i X_i/\Gamma_i, [q_i])$ converges in the pmGH sense (up to subsequence and after renormalizing the measure) to an $RCD^{\ast}(0,N)$ space of rectifiable dimension strictly greater than $m$.
\end{center}
By Theorem \ref{gigli} any partial limit splits as $Y = \mathbb{R}^m \times W$, so all we need  is to  rule out  $W$ being a point. If $Y = \mathbb{R}^m$, then by Proposition \ref{gigli-corollary} after taking a subsequence we can assume that the sequence $(\alpha_i \lambda_i X_i,q_i^{\prime})$ converges in the pGH sense to a space of the form $(\mathbb{R}^m \times Z, (0,z))$ with $(Z,z)$ a pointed proper geodesic space and $\Gamma_i$ converges equivariantly to a closed group $\Gamma \leq Iso (\mathbb{R}^m \times Z)$ in such a way that the $\Gamma$-orbits coincide with the $Z$-fibers. Notice however, that $1 \in \sigma (\Gamma_i , \alpha_i \lambda_i  X_i, q_i^{\prime} )$ for each $i$ by construction, but by Remark \ref{continuous-spectrum}, $1 \notin \sigma (\Gamma , \mathbb{R}^m \times Z , (0,z))$, contradicting Proposition \ref{spec-cont}. 
\end{proof}

\section{Diameter of Compact Universal Covers}\label{section-4}

In this Section we prove Theorem \ref{diam}, which follows easily from Theorem \ref{infinite} below. Our proof is significantly different from the one for the smooth case by Kapovitch--Wilking and is mostly geometric-group-theoretical based on the structure of approximate groups by Breuillard--Green--Tao. We begin by proving a series of Lemmas that reduce Theorem \ref{infinite} to the case when $\Gamma $ is a simply connected nilpotent Lie group.

\begin{thm}\label{infinite}
 Let $\Gamma$ be a non-compact $r$-dimensional Lie group with a left invariant metric, $\Gamma_i$ a sequence of discrete groups, and $f_i : \Gamma _i \to \Gamma$ a clean discrete approximation. Then $\Gamma_i$ is infinite for $i$ large enough.
\end{thm}

\begin{lem}\label{clean-cci}
 Let $\Gamma$ be a Lie group with a left invariant metric, $\Gamma_0 \triangleleft $ $ \Gamma$ the connected component of the identity, $\Gamma_i$ a sequence of discrete groups, and $f_i : \Gamma_i \to \Gamma$ a clean discrete approximation. Let $r > 0 $ be such that $B_e(r,\Gamma ) \subset \Gamma_0$. For each $i \in \mathbb{N}$ let $G_i \leq \Gamma_i$ denote the subgroup generated by $f_i^{-1}(B_e(r,\Gamma ))$. Then for large enough $i$, there is a surjective morphism  $ \Gamma_i / G_i \to \Gamma / \Gamma_0$.
\end{lem}

For the proof of Lemma \ref{clean-cci} we will need the fact below which follows immediately from the definition of determining set.

\begin{prop}\label{obvious}
 Let $G, H$ be groups, $S \subset G$ a determining set, and a function 
\[ \varphi : S \cup S^{-1} \to H\]
such that $\varphi (s_1s_2) = \varphi (s_1) \varphi (s_2)$ for all $s_1, s_2 \in S \cup S^{-1}$ with $s_1s_2 \in S \cup S^{-1}$. Then $\varphi$ extends to a unique group morphism $\tilde{\varphi} : G \to H$.
\end{prop}

\begin{proof}[Proof of Lemma \ref{clean-cci}:] 
Let $S_0 \subset K_0 \subset \Gamma$, $S_i \subset \Gamma_i$ be the sets given by the definition of clean discrete approximation. We then define a map $\varphi_i : S_i \cup S^{-1}_i \to \Gamma /\Gamma_0$ as the composition of $f_i$ with the projection $\Gamma \to \Gamma/\Gamma_0$. Certainly, if $i$ is large enough, $\varphi_i (s_1s_2) = \varphi_i (s_1) \varphi _i(s_2)$ for all $s_1 , s_2 \in S_i \cup S^{-1}_i$ with $s_1 s_2 \in S_i \cup S^{-1}_i$,  as the space $\Gamma / \Gamma_0$ is discrete, and $f_i$ is uniformly close to being a morphism when restricted to $(S_i \cup S_i^{-1})^2 \subset f_i^{-1}((K_0  \cup K_0^{-1} )^3)$. By Proposition \ref{obvious}, we then obtain a group morphism $\tilde{\varphi} _i: \Gamma_i \to \Gamma / \Gamma_0$ for all large $i$. 

For large enough $i$, 
\[ S_0 \subset \bigcup_{x \in  f_i (S_i)}B_x(r,\Gamma ),\]
 so $f_i(S_i)$ intersects each component of $S_0\Gamma_0$. Since the latter by hypothesis generates $\Gamma$, we conclude that $\varphi_i(S_i)$ generates $\Gamma/\Gamma_0$ and $\tilde{\varphi}_i$ is surjective. On the other hand, $G_i$ is clearly in the kernel of $\tilde{\varphi}_i$, so one gets a surjective morphism $ \Gamma _i / G_i \to \Gamma / \Gamma_0$.
\end{proof}

Proposition \ref{determining-characterization} below characterizes determining sets in terms of discrete homotopies.

\begin{defn}
\rm Let $G$ be a group and $S \subset G$ a symmetric generating set. An $S$\textit{-curve} in $G$ is a sequence of elements $ (g_0, \ldots , g_n) $ with $g_{j-1}^{-1}g_{j} \in S   $ for each $j \in \{1, \ldots , n \}$. If $g_0=g_n$, the $S$-curve $ (g_0, \ldots , g_n) $ is said to be \textit{closed based at }$g_0$. 

For two $S$-curves $g= (g_0, \ldots, g_n)$, $h=(h_0, \ldots , h_m)$ in $G$, their \textit{uniform distance} $d_U^S( g,h )$ is defined as the infimum $r \in \mathbb{N}$ such that there is a relation $ R \subset \{ 0, \ldots , n \} \times \{ 0, \ldots , m \}$  such that:
\begin{itemize}
    \item For all $i \in \{0, \ldots , n \}$, $j \in \{0, \ldots , m \}$, there are $i^{\prime} \in \{0, \ldots , n \}$, $ j^{\prime} \in \{0, \ldots , m\}$ with $(i,j^{\prime}), (i^{\prime }, j) \in R$.
    \item If $(i_1, j_1),(i_2, j_2) \in R $ and $i_1<i_2$, then $j_1 \leq j_2$.
    \item $g_i^{-1}h_j \in S^r$ for all $(i,j ) \in R$.
\end{itemize}
\end{defn}
\begin{prop}\label{homotopy-replacement}
 Let $G$ be a group, $S \subset G$ a symmetric generating set containing the identity, and $g=(g_0, \ldots , g_n)$, $h=(h_0, \ldots , h_m)$ two closed $S$-curves based at $e$ with $d_U^S(g,h) \leq N$. For each $j \in \{1, \ldots , n \}$, $\ell \in \{1, \ldots , m\}$ set $s_j = g_{j-1}^{-1}g_j$ and $t_{\ell} = h_{\ell -1}^{-1} h_{\ell}$. Then starting with the word $s_1\ldots s_n$ one can obtain the word $t_1 \ldots t_m$ as words with letters in $S^{N+1}$ by a finite number of substitutions of the form $\theta \to \theta_1\theta_2$ or $\theta_1\theta_2 \to \theta$ with $\theta, \theta_1, \theta_2 \in S^{N+1}$ satisfying $\theta = \theta_1\theta_2$. 
\end{prop}
\begin{proof}
Let $R \subset \{ 0, \ldots , n \} \times \{ 0, \ldots , m \}$ be such that 
    \begin{itemize}
    \item For all $i \in \{0, \ldots , n \}$, $j \in \{0, \ldots , m \}$, there are $i^{\prime} \in \{0, \ldots , n \}$, $ j^{\prime} \in \{0, \ldots , m\}$ with $(i,j^{\prime}), (i^{\prime }, j) \in R$.
    \item If $(i_1, j_1),(i_2, j_2) \in R $ and $i_1<i_2$, then $j_1 \leq j_2$.
    \item $g_i^{-1}h_j \in S^N$ for all $(i,j ) \in R$.
\end{itemize} 
For each $(i,j )\in R$, set $u_{i,j} : = g_i^{-1}h_j \in S^{N}$. We claim that for each $(i,j)\in R$, the word $s_1 \ldots s_i u_{i,j}t_{j+1} \ldots t_m$ can be obtained from $s_1\ldots s_n$ as words with letters in $S^{N+1}$ by a finite number of substitutions of the form $\theta \to \theta_1\theta_2$ or $\theta_1\theta_2 \to \theta$ with $\theta, \theta_1, \theta_2 \in S^{N+1}$ satisfying $\theta = \theta_1\theta_2$. 

We verify this claim by reverse induction on $i+j$. As $u_{n,m} = e$, the claim is trivial for $(i,j ) = (n,m)$. Then for each $(i,j) \in R$ with $(i,j) \neq (n,m)$, there are three cases to consider:

\begin{itemize}
    \item $(i+1,j)\in R$, in which case, $s_{i+1}u_{i+1,j} = u_{i,j}$.
    \item  $( i,j+1) \in R$, in which case, $u_{i,j+1}  = u_{i,j}t_{j+1}$.
    \item  $(i+1,j+1) \in R$, in which case, one can perform the substitutions $s_{i+1}u_{i+1, j+1} \to $ $(g_i^{-1}h_{j+1})  \to u_{i,j}t_{j+1}$, where $(g_i^{-1}h_{j+1}) \in S^{N+1}$ is considered as a one-letter word.
\end{itemize}

\begin{figure}
\centering
\psfrag{a}{$s_{i+1}$}
\psfrag{b}{$u_{i+1,j}$}
\psfrag{c}{$u_{i,j}$}
\psfrag{d}{$t_{j+1}$}
\psfrag{e}{$u_{i,j+1}$}
\psfrag{f}{$u_{i,j}$}
\psfrag{g}{$s_{i+1}$}
\psfrag{h}{$t_{j+1}$}
\psfrag{i}{$t_{j+1}$}
\psfrag{j}{$s_{i+1}$}
\psfrag{k}{$u_{i+1, j+1 }$}
\psfrag{l}{$u_{i,j}$}
\psfrag{m}{\small $g_i^{-1}h_{j+1}$}
\includegraphics[width=0.9\textwidth]{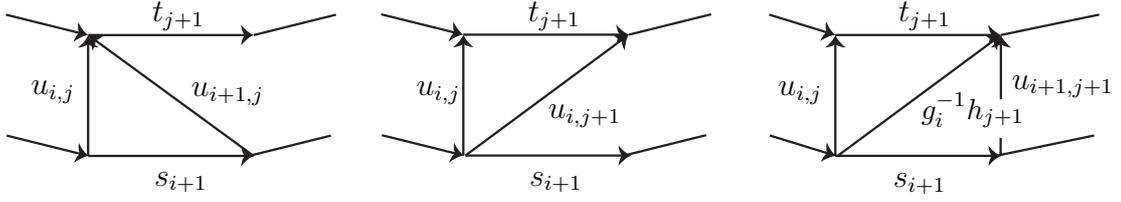}
\caption{One can obtain the word $s_1 \ldots s_i u_{i,j}t_{j+1} \ldots t_m$ from $s_1 \ldots s_i s_{i+1} u_{i+1,j}t_{j+1} \ldots t_m$, $s_1 \ldots s_i u_{i,j+1}t_{j+2} \ldots t_m$, or $s_1 \ldots s_i s_{i+1}u_{i+1,j+1}t_{j+2} \ldots t_m$, respectively.}\label{replace}
\end{figure}

In either case, one can combine the induction hypothesis with the above substitutions to obtain $s_1 \ldots s_i u_{i,j}t_{j+1} \ldots t_m$. As $u_{0,0} = e$ (see Figure \ref{replace}), the result follows from the claim applied to $(0,0) \in R$.
\end{proof}

\begin{prop}\label{determining-characterization}
 Let $G$ be a group, $S \subset G$ a symmetric generating set containing the identity, then 
\begin{enumerate}
    \item If $S$ is a determining set, then for any closed $S$-curve $\gamma_0$  based at $e$, there is a sequence $\gamma_1, \ldots , \gamma_n $ of closed $S$-curves based at $e$ with $d_U^S(\gamma_{j-1}, \gamma_j ) \leq 1$ for each $j \in \{ 1, \ldots , n\}$ and $\gamma_n = (e)$.\label{det-1}
    \item If for each closed $S$-curve $\gamma _0 $  based at $e$, there is a sequence of closed $S$-curves $\gamma_1, \ldots , \gamma_n$  based at $e$ with $d_U^S(\gamma_{j-1}, \gamma_j ) \leq N$ for each $j \in \{ 1, \ldots , n\}$ and $\gamma_n = (e)$, then $S^{N+1}$ is a determining set.\label{det-2}
\end{enumerate}
\end{prop}

\begin{proof} \ref{det-1} Let $\gamma_0 = (g_0, \ldots , g_m)$ be a closed $S$-curve  based at $e$, and consider $s_j := g_{j-1}^{-1}g_j \in S$ for $j \in \{1, \ldots , m\}$. As $s_1 \cdots s_m = e$ and $S$ is determining, we can reduce the word $s_1 \ldots  s_m$ to the trivial word by a finite number of substitutions of the form $ t \to  t_1 t_2 $ or $t_1t_2 \to t$ with $t, t_1, t_2 \in S$ and $t = t_1t_2$. Let $w_0, w_1, \ldots , w_k$ be the words obtained throughout this process with $w_0 = s_1\ldots s_m$ and $w_k = e$.

For $j \in \{1, \ldots , k\}$, set $\gamma_j := ( e, w_j^1, \ldots, w_j^1 \cdots w_j^{m_j})$, where $w_j = w_j^1 \ldots w_j^{m_j}$.  It is straightforward to check that $\gamma_k = (e)$ and  for each $j \in \{1, \ldots , k \}$,   $\gamma_j$  is a closed $S$-curve based at $e$ with $d_U^S(\gamma_{j-1}, \gamma_j)\leq 1$.

\ref{det-2} Consider a word $t_1\cdots t_n $ with letters in $S^{N+1}$ representing the trivial element in $G$. For each $j \in \{1, \ldots , n\}$, take $s_1, \ldots , s_{N+1} \in S$ such that $s_1\cdots s_{N+1} = t_j$, and set $u_{\ell} := s_1\cdots s_{\ell}$ for $\ell \in \{1, \ldots , N \}$. As $u_{N+1} = t_j$ and $u_{\ell -1}s_{\ell} = u_{\ell}$ for each $\ell$, one can expand the one letter word $t_j$ to the word $s_1\ldots s_{N+1}$ by a finite number of substitutions $\theta \to \theta_1\theta_2 $ or $\theta _1 \theta_2 \to \theta$ with $\theta, \theta_1 , \theta_2 \in S^{N+1}$ and $\theta = \theta_1 \theta_2$. By doing this for each $j$ we can, by a finite number of substitutions of the above form, expand the word we began with ($t_1 \cdots t_n$) to a word $\sigma_1 \ldots \sigma_m$ with letters in $S$. This word corresponds to a closed $S$-curve $\gamma_0 = (e, \sigma_1, \sigma_1\sigma_2 , \ldots, \sigma_1 \cdots \sigma_m)$. 

By hypothesis, there are closed $S$-curves $\gamma_1, \ldots , \gamma_k$ based at $e$ with $d_U^S(\gamma_{j-1}, \gamma_j ) \leq N$ for each $j \in \{ 1, \ldots , k\}$ and $\gamma_k = (e)$. For each $j\in \{0, 1, \ldots , k \}$, set $\gamma_j = (e, \sigma _{j,1},\sigma _{j,1}\sigma _{j,2}, \ldots, \sigma _{j,1} \cdots \sigma _{j,m_j})$.

Then by repeated applications of Proposition \ref{homotopy-replacement}, one can obtain from $\sigma_1 \ldots \sigma_m$ the trivial word  as words with letters in $S^{N+1}$ by a finite number of substitutions of the form $\theta \to \theta_1\theta_2$ or $\theta_1\theta_2 \to \theta$ with $\theta, \theta_1, \theta_2 \in S^{N+1}$ satisfying $\theta = \theta_1\theta_2$, meaning that $S^{N+1}$ is determining. 
\end{proof}

\begin{prop}\label{determining-inheritance}
 Let $G$ be a group, $S \subset G$ a symmetric determining set, and $H \triangleleft G$ a normal subgroup with $[G:H] < \infty$. If $S$ intersects each $H$-coset in $G$, then $(S^3 \cap H)^2$ is a determing set of $H$.
\end{prop}
\begin{proof}
  It is easy to see that $ T: = S^3 \cap H  $ generates $H$ (\cite{hall}, Lemma 7.2.2). Choose $\{ x_1, \ldots , x_M \} \subset S$ a set of $H$-coset representatives with $x_1 = e$ and let $\eta: G\to H$ be the map that sends $g$ to $gx_j^{-1}$, where $x_j$ is the only element of $\{ x_1, \ldots , x_M \} $ with $gx_j^{-1} \in H$. By our construction of $T$, if $g_0, g_1 \in G$ are such that $g_0^{-1}g_1 \in S$, then $\eta(g_0)^{-1}\eta(g_1) \in T$. 
  
  This implies that for each $S$-curve $w =  (g_0, \ldots , g_n) $ in $G$, the sequence $ \tilde{\eta} (w) : = (\eta(g_0), \ldots , \eta (g_n)) $ is a $T$-curve in $H$, and if two $S$-curves $w_0, w_1$ have the same endpoints and $d_U^S(w_0,w_1) \leq 1$, then $\tilde{\eta}(w_0), \tilde{\eta}(w_1)$ also have the same endpoints and $d_U^T(\tilde{\eta}(w_0), \tilde{\eta}(w_1))\leq 1$.

On the other hand, for each $T$-curve $(h_0, \ldots , h_n )$ in $H$ there is an $S$-curve $(g_0, \ldots, g_{3n})$ in $G$ with $g_{3j} = h_j$ for each $j \in \{0, \ldots , n\}$. Then it is easy to check that the $d_U^T$-distance between $\tilde{\eta}(g_0, \ldots , g_{3n})$ and $(h_0, \ldots , h_n)$ is $\leq 1$.

From the above, for any closed $T$-curve $\gamma_0$ based at $e$ in $H$, there is a closed $S$-curve $\alpha_0$ based at $e$ in $G$ with  $d_U^T(\tilde{\eta}(\alpha_0) , \gamma_0 ) \leq 1$. Then by Proposition \ref{determining-characterization} \ref{det-1}, there is a sequence of closed $S$-curves $\alpha_1 \ldots , \alpha_n$ based at $e$ in $G$ with $d_U^{S}(\alpha_{j-1}, \alpha_j) \leq 1$ for each $j\in \{1, \ldots , n\}$ and $\alpha_n = (e)$.

Applying $\tilde{\eta}$ we obtain a sequence of closed $T$-curves $\tilde{\eta}(\alpha_0), \ldots ,  \tilde{\eta}(\alpha_n)$ based at $e$ in $H$ with $d_U^T(\tilde{\eta}(\alpha_0) ,$ $ \gamma_0 ) \leq 1$, 
\[ d_U^T( \tilde{\eta}(\alpha_{j-1}), \tilde{\eta}(\alpha_j) ) \leq 1\text{ for each }j \in \{1, \ldots , n\},\] 
and $\tilde{\eta}(\alpha_n) = (e)$, which by Proposition \ref{determining-characterization} \ref{det-2} implies that $T^2$ is a determining set in $H$.
\end{proof}
\begin{lem}\label{cci-reduction}
 Under the conditions of Lemma \ref{clean-cci}, if $\Gamma / \Gamma_0$ is finite then there is a clean discrete approximation $h_i : G_i \to \Gamma_0$.
\end{lem}

\begin{proof}
We split the situation in two cases. In either case, it is straightforward to check that $h_i$ is a discrete approximation.

\begin{center}
    \textbf{Case 1:} $\Gamma_0$ is compact.
\end{center}
For $i$ large enough, $G_i$ is finite and $G_i = f_i^{-1}(\Gamma_0)$, so we define $h_i$ as $f_i \vert_{G_i}$.
\begin{center}
    \textbf{Case 2:} $\Gamma_0$ is non-compact.
\end{center}
Consider a sequence $g_i \in \Gamma_0$ with $d(e,g_i) \to \infty $. Then define $h_i : G_i \to \Gamma_0$ as
\begin{equation*}
h_i (g):=
\begin{cases}
f_i (g) & \text{if }\, f_i(g) \in \Gamma_0 , \\
g_i & \text{if }\, f_i(g) \notin \Gamma_0.
\end{cases}
\end{equation*}
Let $S_0\subset  K_0 \subset \Gamma$, $S_i \subset \Gamma_i$ be the sets given by the definition of clean discrete approximation. Then by Proposition \ref{determining-inheritance}, $S_i^{\prime} : = (S_i ^3 \cap G_i )^2 $ is a determining set. Since the maps $h_i$ are approximate morphisms, we have
\[ h_i^{-1}(S_0 \cap \Gamma_0)   \subset S_i^{\prime} \subset    h_i^{-1}(K^7_0 \cap \Gamma_0 )  \] 
for large enough $i$, showing that the discrete approximation $h_i:G_i \to \Gamma_0$ is clean.
\end{proof}

\begin{rem}\label{geodesic}
\rm Notice that if $\Gamma$ is a connected Lie group with a proper left invariant metric $d$, then for any left invariant geodesic metric $d_1$ in $\Gamma$ and any $R>0$ there is $R^{\prime}>0$ with  $B_e^{d_1}(R, \Gamma )  \subset B_e^{d}(R^{\prime} , \Gamma )  $. Hence if one has a discrete approximation $f_i : \Gamma_i \to (\Gamma, d)$,  the maps $f_ i : \Gamma _i \to (\Gamma , d_1)$ also form a discrete approximation.
\end{rem}

\begin{lem}\label{sc-reduction}
 Let $\Gamma$ be a connected Lie group with a left invariant geodesic metric, $\Gamma_i$ a sequence of discrete groups, $f_i : \Gamma_i \to \Gamma$ a clean discrete approximation, and $K_0 \leq \Gamma$ a compact subgroup. Then $K_0 $ is contained in the center of $\Gamma$, and there is a clean discrete approximation $h_i : \Gamma _i \to \Gamma / K_0$. 
\end{lem}

\begin{proof}
 By Theorem \ref{bgt}, $\Gamma$ is nilpotent, and since compact subgroups of connected nilpotent Lie groups are central (\cite{zamora-fglahs}, Section 2.6), $K_0$ is. As $K_0$ is compact normal, the metric $d^{\prime}$ in $\Gamma /K_0$ given by
 \[     d^{\prime}([a],[b]) : = \inf_{k \in K_0} d(a,kb)          \]
 is proper, geodesic, and left invariant. Then define $h_i : \Gamma_i \to \Gamma / K_0$ as the composition of $f_i$ with the projection $\Gamma \to \Gamma /K_0$. All conditions in the definition of clean discrete approximation then follow for $h_i$ from the corresponding ones of $f_i$, as the projection $\Gamma \to \Gamma / K_0$ is 1-Lipschitz and proper. 
\end{proof}

We will also need the following elementary fact about nilpotent Lie groups (\cite{corwin-greenleaf}, Section 1.2).

\begin{prop}\label{nilpotent-compact}
 Let $\Gamma$ be a connected nilpotent Lie group. Then there is a unique normal compact subgroup $K_0 \leq \Gamma$ such that the quotient $\Gamma / K_0 $ is simply connected.  
\end{prop}

\begin{prop}\label{isomorphism}
 Let $G_0, G_1$ be groups, $S_0 \subset G_0$, $S_1 \subset G_1$ determining sets, and $\theta : G_0 \to G_1$ a morphism with
\begin{itemize}
    \item $S_1 \subset \theta(S_0)$.
    \item $Ker(\theta) \cap (S_0 \cup S_0^{-1})^{2} = \{ e \}$.
\end{itemize}
Then $\theta$ is an isomorphism.
\end{prop}
\begin{proof}
 Construct $\psi : S_1 \to S_0$ as $\psi (s) = \theta^{-1}(s)$. By the first condition such element exists, and by the second condition the map is well defined. Since $S_1$ is determining in $G_1$, there is an inverse map $\tilde{\psi} : G_1 \to G_0$.
\end{proof}

\begin{proof}[Proof of Theorem \ref{infinite}:] 
If $\Gamma$ has infinitely many connected components, then by Lemma \ref{clean-cci}, $\Gamma_i$ is infinite for large enough $i$, so we can assume $\Gamma$ has finitely many connected components.  By Lemma \ref{cci-reduction}, we can assume that $\Gamma$ is connected. By Remark \ref{geodesic} we can assume its metric is geodesic, and  by  Lemma \ref{sc-reduction} and Proposition \ref{nilpotent-compact}, we can further assume $\Gamma$ is simply connected.

By Theorem \ref{bgt}, there is $C>0$, a sequence of small subgroups $H_i \triangleleft \Gamma_i$, and rank $r$ nilprogressions $P_i \subset \Gamma_i^{\prime} : =   \Gamma_i / H_i$ in $C$-regular form with thick$(P_i) \to \infty$ and satisfying that for each $\varepsilon \in (0,1]$ there is $\delta > 0 $ with
\begin{itemize}
    \item $\tilde{f}_i^{-1}(B_e(\delta , \Gamma )) \subset G(\varepsilon P_i )$ for $i$ large enough,
    \item $ \tilde{f}_i ( G(\delta P _i) ) \subset B_e(\varepsilon , \Gamma)  $ for $i$ large enough, 
\end{itemize} 
where $\tilde{f}_i : \Gamma _ i ^{\prime } \to \Gamma$ is the discrete approximation given by Proposition \ref{small-quotient}. 

By iterated applications of Theorem \ref{malcev}, for $i$ large enough there are infinite groups $ (\mathbb{Z}^r , \ast_{P_i} )$, symmetric determining sets $T_i \subset T_i^{\prime} \subset (\mathbb{Z}^r , \ast_{P_i} )$, $\delta ^{\prime} > \delta > 0 $, and  surjective morphisms $\theta_i: (\mathbb{Z}^r , \ast_{P_i} ) \to \Gamma_i^{\prime}$ such that $Ker ( \theta_i) \cap ( T_i^{\prime })^2 = \{ e \}$ and  
\[   \tilde{f}_i^{-1} (B_e(\delta, \Gamma )) \subset \theta _i ( T_i) \subset  \tilde{f}_i^{-1} (B_e(\delta ^{\prime}/8, \Gamma )) \subset  \tilde{f}_i^{-1} (B_e(\delta ^{\prime}, \Gamma )) \subset  \theta_i (T_i^{\prime}) . \]
It then follows by induction on $N \in \mathbb{N}$ that if $a_{1,i}, \ldots , a_{N,i} \in \theta_i^{-1} (\tilde{f}_i^{-1} ( B_e(\delta^{\prime}/3, \Gamma ) ) ) \cap T_i^{\prime} $ and $b_i \in T_i$ are sequences with $\tilde{f}_i(\theta_i (b_i) ) \to e$, then 
\begin{align}
\label{induct-approximate}
\begin{split}
  a_{1,i} \cdots &a_{N,i} b_i(  a_{1,i} \cdots a_{N,i})^{-1}  \in T_i\text{ for large enough }i,
\\
 &\tilde{f}_i (\theta_i (a_{1,i} \cdots a_{N,i} b_i ( a_{1,i} \cdots a_{N,i})^{-1})) \to e. 
\end{split}
\end{align}

Also notice that if one has two sequences $g_i, h_i \in (\mathbb{Z}^r , \ast_{P_i} )$ such that $g_i, h_i \in T_i$ for large enough $i$,  and $\tilde{f}_i (\theta_i(g_i )) \to e , \tilde{f}_i  ( \theta_i(  h_i)) \to e$, then for large enough $i$ one also has 
\begin{equation}\label{product}
    g_ih_i \in T_i\text{ and }\tilde{f}_i (\theta_i(g_i h_i)) \to e.
\end{equation}

Let  $S_i \subset \Gamma_i^{\prime}$ be the sets given by the definition of clean discrete approximation. Then by connectedness of $\Gamma$, there is $M \in \mathbb{N}$ such that for $i$ large enough, $S_i \cup S_i^{-1} \subset \theta_i(T_i)^{M} $.
\begin{center}
    \textbf{Claim:} $Ker(\theta_i) \cap T_i ^{2M} = \{ e \}$ for $i$ large enough.
\end{center}
Working by contradiction, after taking a subsequence we can find sequences $ x_m^i \in T_i$ with $m \in \{1, \ldots , 2M \} $, $i \in \mathbb{N}$, such that $x_1^i \cdots x_{2M}^i \in Ker(\theta_i) \backslash \{ e \}$ for each $i$. After further taking a subsequence, we can assume there are elements $y_m \in B_e(  \delta ^{\prime} /4 , \Gamma )$ with $\tilde{f}_i(\theta_i(x_m^i)) \to y_m$ as $i \to \infty$ for $m\in \{1, \ldots , 2M \}$. 

Since $\Gamma$ is simply connected, there is a finite sequence of chains $\{ y_{\ell ,1}, \ldots , y_{\ell , k_{\ell} } \} \subset B_e( \delta ^{\prime} /4 , \Gamma )  $ with $\ell \in \{ 0 , \ldots , n \}$ such that (see Figure \ref{induction-figure}):
\begin{itemize}
    \item $k_0 = 2M$ and $y_{0, m} = y_m$ for all $m \in \{ 1, \ldots , 2M \}$.
    \item $k_n = 1$ and $y_{n,1} = e$.
    \item  For each $\ell \in \{1, \ldots, 2M \}$, there is $j_{\ell} \in \{ 1, \ldots , k_{\ell} \}$ such that the words $\{ y_{\ell ,1}, \ldots , y_{\ell , k_{\ell} } \}$ and $\{ y_{\ell-1 ,1}, \ldots , y_{\ell-1 , k_{\ell-1} } \}$ satisfy that $y_{\ell, m}= y_{\ell -1 , m } $ for $m < j_{\ell} $ and either:
    \begin{enumerate}
        \item  $y_{\ell, m} = y_{\ell-1, m+1}$ for $m > j_{\ell}$, and $y_{\ell,j_{\ell}}=y_{\ell -1, j_{\ell}} y_{\ell -1 , j_{\ell}+1}$. \label{triangle1}
        \item  $y_{\ell, m} = y_{\ell-1, m-1}$ for $m > j_{\ell}+1$, and $y_{\ell,j_{\ell}} y_{\ell, j_{\ell}+1}=y_{\ell -1, j_{\ell}} $. \label{triangle2}
    \end{enumerate}
\end{itemize}
For each $m \in \{1, \ldots , 2M \} $, set $x_{0,m}^i : = x_m^i$, and for each $\ell \in \{ 1, \ldots , n \}$, $m \in \{ 1, \ldots , k_{\ell} \}$, choose a sequence $x_{\ell , m}^i \in T_i^{\prime}$ such that $\tilde{f}_i(\theta_i (x_{\ell , m}^i)) \to y_{\ell, m}$ as $i \to \infty$. We take them so that for all $i$, one has $x_{\ell, m}^i = x_{\ell - 1 , m}^i$ for $m < j_{\ell}$, and 
\begin{itemize}
    \item   if $ j_{\ell}$ satisfies \ref{triangle1}, then $x_{\ell, m}^i = x_{\ell -1 , m+1} ^i$ for $m > j_{\ell}$.
    \item  if $ j_{\ell}$ satisfies \ref{triangle2},  then $x_{\ell, m}^i = x_{\ell -1 , m-1} ^i$ for $m > j_{\ell} + 1$.
\end{itemize}
 Notice that 
\begin{equation}\label{base-approximate}
    x_{n,1}^i \in T_i\text{ for large enough }i \text{ and } \tilde{f}_i (\theta _i (x_{n,1}^i) ) \to e.
\end{equation}
We will show by reverse induction on $\ell$ that
\begin{equation}\label{cycle-infinite}
w_{\ell}^i :=  x_{\ell , 1}^i \cdots x_{\ell , k_{\ell }} ^i \in T_i \text{ for large enough }i  \text{ and }    \tilde{f}_i (\theta_i (x_{\ell , 1}^i \cdots x_{\ell , k_{\ell }} ^i )) \to e  ,
\end{equation}
the base of induction being Equation \ref{base-approximate}.

If $j_{\ell}$ satisfies \ref{triangle1}, then  $w_{\ell -1}^i = x_{\ell-1, 1}^i \cdots x^i_{\ell-1,k_{\ell-1}} $ is obtained from $w_{\ell}^i $ by multiplying it by $ (x_{\ell, j_{\ell}+1}^i \cdots x^i_{\ell,k_{\ell}})^{-1} ( (x^i_{\ell, j_{\ell}})^{-1}  x^i_{\ell-1, j_{\ell}} x^i_{\ell -1, j_{\ell}+1}  ) (x_{\ell-1, j_{\ell}+2}^i \cdots x^i_{\ell-1,k_{\ell-1}} ) $. Putting $N := k_{\ell} - j_{\ell}$, $a_{m,i} := x_{\ell ,  j_{\ell}+m}^i = x_{\ell -1 , j_{\ell}+m +1}^i$, and $b_i = (x^i_{\ell, j_{\ell}})^{-1}  x^i_{\ell-1, j_{\ell}} x^i_{\ell -1, j_{\ell}+1}$, the induction step follows from Equations  \ref{induct-approximate} and \ref{product}. 

If $j_{\ell}$ satisfies \ref{triangle2}, then  $w_{\ell -1}^i $ is obtained from $w_{\ell}^i $ by multiplying it by 

$ (x_{\ell, j_{\ell}+2}^i \cdots x^i_{\ell,k_{\ell}})^{-1} ( (x^i_{\ell, j_{\ell}} x^i_{\ell , j_{\ell} + 1} )^{-1}  x^i_{\ell-1, j_{\ell}} ) (x_{\ell-1, j+1}^i \cdots x^i_{\ell-1,k_{\ell-1}} ) $, and just like in the first case, the induction step follows from Equations \ref{induct-approximate} and \ref{product}.

\begin{figure}
\centering
\psfrag{a}{$y_1$}
\psfrag{b}{$y_2$}
\psfrag{c}{$y_3$}
\psfrag{d}{$y_4$}
\psfrag{e}{$y_5$}
\psfrag{f}{$y_6$}
\psfrag{g}{$y_7$}
\psfrag{h}{$y_8$}
\psfrag{i}{$y_{\ell-1,j_{\ell}}$}
\psfrag{j}{$y_{\ell-1, j_{\ell}+1}$}
\psfrag{k}{$y_{\ell, j_{\ell}}$}
\psfrag{l}{$y_{\ell-1,j_{\ell}}$}
\psfrag{m}{$y_{\ell, j_{\ell}}$}
\psfrag{n}{$y_{\ell, j_{\ell}+1}$}
\includegraphics[width=0.7\textwidth]{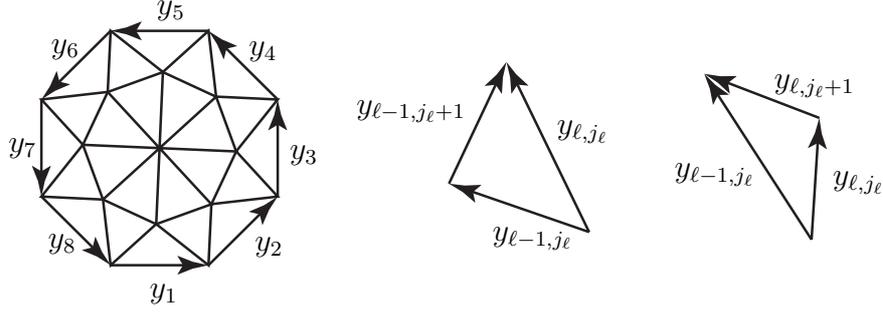}
\caption{For each $\ell$, there is $j_{\ell}$ such that either $y_{\ell-1,j_{\ell}}y_{\ell -1,j_{\ell}+1}= y_{\ell , j_{\ell}}$ or $y_{\ell-1, j_{\ell}} = y_{\ell, j_{\ell}}y_{\ell, j_{\ell}+1}$.}\label{induction-figure}
\end{figure}

Therefore, $x_{0,1}^i \cdots x_{0,2M}^i \in  Ker (\theta_i ) \cap T_i = \{ e \}$ for $i$ large enough. However, this contradicts our choice of $x_j^i=x_{0,j}^i$, proving our claim.
 
By Remark \ref{s-and-sm}, Proposition \ref{isomorphism}, and the claim, $\theta_i$ is an isomorphism and $\Gamma_i^{\prime} $ is infinite for $i$ large enough.
\end{proof}

\begin{customthm}{2}
 Let $(X,d,\mathfrak{m})$ be an $RCD^{\ast}(K,N;D)$ space. If the universal cover $(\tilde{X},\tilde{d},\tilde{\mathfrak{m}})$ is compact, then diam$( \tilde{X} )\leq \tilde{D}( N,K, D)$.    
\end{customthm}

\begin{proof}[Proof of Theorem \ref{diam}:]
Assuming the contrary, we get a sequence $X_i$ of compact $RCD^{\ast}(K,N;D)$ spaces such that their universal covers $\tilde{X}_i$ satisfy diam$(\tilde{X}_i) \to \infty$. After taking points $p_i \in \tilde{X}_i$ and a subsequence, we can assume that the sequence $(\tilde{X}_i , p_i)$ converges in the pGH sense to a pointed $RCD^{\ast}(K,N)$ space $(X,p)$ and the sequence $\pi_1(X_i)$ converges equivariantly to a closed group $\Gamma\leq Iso (X)$. Since $X$ is non-compact, and $X/\Gamma$ is compact, we conclude that $\Gamma$ is non-compact. By Corollary \ref{clean-uc}, we have a clean discrete approximation $f_i : \pi_1(X_i) \to \Gamma$. However, the groups $\pi_1(X_i)$ are finite, contradicting Theorem \ref{infinite}.
\end{proof}

\section{Uniform Virtual Abelianness}\label{section-5}

In this section we prove Theorems \ref{ua-1} and \ref{ua-2}. The main ingredient is the following result, which implies that non-collapsing sequences of $RCD^{\ast}(K,N)$ spaces cannot have small groups of measure preserving isometries.

\begin{thm}\label{nss}
 Let $(X_i, d_i, \mathfrak{m}_i,p_i) $ be a sequence of pointed $RCD^{\ast}(K,N)$ spaces of rectifiable dimension $n$ and let $H_i \leq Iso (X_i)$ be a sequence of small subgroups in the sense of Remark \ref{small-2} acting by measure preserving isometries. Assume the sequence $(X_i, d_i, \mathfrak{m}_i,p_i)$ converges in the pmGH sense to an $RCD^{\ast}(K,N)$ space $(X, d, \mathfrak{m},p) $ of rectifiable dimension $n $. Then $H_i$ is trivial for large enough $i$.
\end{thm}

\begin{proof}
Working by contradiction, after taking a subsequence and replacing $H_i$ by its closure, we can assume that $H_i$ is non-trivial and compact for each $i$. 

 Let $q \in X$ be an $n$-regular point and $p_i^{\prime} \in X_i$ be such that $p_i^{\prime}$ converges to $q$. Notice that if $\lambda_i \to \infty$ slowly enough, then $(\lambda_i X_i , p_i^{\prime})$ converges (after renormalizing the measure) in the pmGH sense to $(\mathbb{R}^{n} , d^{\mathbb{R}^{n}},\mathcal{H}^{n}, 0 )$ and the groups $H_i$ are still small subgroups when acting on $(\lambda_iX_i, p_i^{\prime} )$, so by Theorem \ref{compact-quotient} the sequence $(\lambda_i X_i / H_i , [p_i^{\prime}])$ converges (after renormalizing the measure on $\lambda_i X_i / H_i$) to $(\mathbb{R}^{n} , d^{\mathbb{R}^{n}},\mathcal{H}^{n}, 0 )$ as well. 

Define $\eta _ i : X_i \to \mathbb{R}$ as 
\[ \eta_i (x) : = \sup _{h \in H_i} d^{\lambda_iX_i} ( hx,x ).    \]
Then by Lemma \ref{no-fixed}, $\mathfrak{m}_i(\eta_i^{-1}(0))=0$, and by Theorem \ref{zoom} there are sequences $q_i \in X_i$ such that $\eta_i(q_i) \neq 0$ and for any $\alpha_i \to \infty$, the sequence $(\alpha_i \lambda_i X_i /H_i, [q_i]) $ converges (possibly after taking a subsequence) in the pGH sense to a space containing an isometric copy of $\mathbb{R}^{n}$. Set $\alpha_i : = 1/ \eta_i (q_i)$. 

By Proposition \ref{gigli-corollary} after taking a subsequence we can assume that $(\alpha_i \lambda_i X_i, q_i)$ converges in the pGH sense to a pointed space $(\mathbb{R}^{n } \times Z, (0,z))$  and $H_i$ converges equivariantly to a group $H \leq Iso \left( \mathbb{R}^{n} \times Z \right)$ in such a way that the $H$-orbits coincide with the $Z$-fibers. By Remark \ref{dim-semicont-2}, $Z$ is a point and $H$ is trivial, but by our choice of $\alpha_i$, $H$ is non-trivial which is a contradiction.
\end{proof}

\begin{customthm}{4}
 Let $(X_i,d_i,\mathfrak{m}_i)$ be a sequence of $RCD^{\ast}(K,N;D)$ spaces of rectifiable dimension $n$ with compact universal covers $(\tilde{X}_i,\tilde{d}_i,\tilde{\mathfrak{m}}_i)$. If the sequence $(\tilde{X}_i,\tilde{d}_i,\tilde{\mathfrak{m}}_i)$ converges in the pointed measured Gromov--Hausdorff sense to some $RCD^{\ast}(K,N)$ space of rectifiable dimension $n$, then there is $C > 0 $ such that for each $i$ there is an abelian subgroup $A_i \leq \pi_1(X_i)$ with index  $[\pi_1(X_i): A_i] \leq C$.
\end{customthm}

\begin{proof}[Proof of Theorem \ref{ua-2}:]
Arguing by contradiction, after taking a subsequence we can assume that $\pi_1(X_i)$ does not admit an abelian subgroup of index $ \leq i$. By Theorem \ref{diam}, $X$ is compact, and after taking a subsequence we can assume the groups $\pi_1(X_i)$ converge equivariantly to a compact group $\Gamma \leq Iso (X)$. By Corollary \ref{clean-uc}, we have a clean discrete approximation $f_i : \pi_1(X_i) \to \Gamma$. By Theorem \ref{turing}, there are small subgroups $H_i \triangleleft \pi_1(X_i)$ for which the quotients $\pi_1(X_i)/H_i $ are uniformly virtually abelian. On the other hand, by Theorem \ref{nss} the groups $H_i$ are eventually trivial, showing the existence of a constant $C$ for which $\pi_1(X_i)$ admits an abelian subgroup $A_i \leq \pi_1(X_i)$ of index $[\pi_1(X_i):A_i] \leq C$.
\end{proof}

For the proof of Theorem \ref{ua-1} we require the following intermediate step.

\begin{thm}\label{Turing-no-compact}
 Let $(X_i,d_i,\mathfrak{m}_i)$ be a sequence of $RCD^{\ast}(0,N ; D)$ spaces and $(\tilde{X}_i,\tilde{d}_i,\tilde{\mathfrak{m}}_i)$ their universal covers. Then one of the following holds:
\begin{itemize}
\item There is a sequence $H_i \leq Iso (\tilde{X}_i )$ of small subgroups of measure preserving isometries such that $H_i \neq 0$ for infinitely many $i$'s.
\item There is $C > 0 $ and abelian subgroups $A_i \leq \pi_1 (X_i) $ with index $[ \pi_1 (X_i)  , A_i] \leq C$.
\end{itemize}
\end{thm}

\begin{proof}
By contradiction, we can assume that any sequence $H_i \leq Iso (\tilde{X}_i)$ of small subgroups of measure preserving isometries is eventually trivial, and for each $i$, the group $\pi_1(X_i)$ does not admit an abelian subgroup of index $\leq i$. Notice that these properties are passed to subsequences.

By Theorem \ref{gigli}, each universal cover $\tilde{X}_i$ splits as a product $Y_i \times \mathbb{R}^{m_i} $ with $Y_i$ a compact $RCD^{\ast}(0, N-m_i)$ space, and after passing to a subsequence, we can assume that the sequence $m = m_i$ is constant. The action of $\pi_1(X_i)$ respects such splitting, so one has natural maps $ \phi _i : \pi_1(X_i) \to Iso (Y_i)$ and $\psi_i: \pi_1 (X_i) \to Iso (\mathbb{R}^m)$ whose images consist of measure preserving isometries. 

\begin{center}
\textbf{Claim: }There is $C_1 > 0 $ and abelian subgroups $Z_i \leq \phi_i (\pi_1(X_i))$ of index $[   \phi_i (\pi_1(X_i)) : Z_i  ] \leq C_1 $.
\end{center}

By Theorem \ref{zamora-2}, the sequence diam$(Y_i)$ is bounded, and after taking a subsequence we can assume the spaces $Y_i$ converge in the mGH sense to a compact $RCD^{\ast}(0, N-m)$ space $Y$ and the groups $Iso (Y_i) $ converge equivariantly to a compact subgroup $\Gamma \leq Iso (Y)$, so we have almost morphisms $f_i : Iso (Y_i) \to \Gamma $ (see Remark \ref{continuous-approximation}). 

By Theorem \ref{gs-lie}, $Iso(Y)$ is a compact Lie group, so by Theorem \ref{pw}, it has a faithful unitary representation $\rho_0 : Iso(Y) \to U(H)$ for some finite dimensional Hilbert space $H$.  Since the maps $f _i $ are almost morphisms, the compositions 
\[ \eta_i : = \rho_0 \circ f_i  : \phi_i (\pi_1(X_i))  \to U(H) \]
are almost representations. The groups $\phi_i (\pi_1(X_i))$, being virtually abelian, are amenable, so by Theorem \ref{kazhdan}, the maps $\eta_i$ are close to actual representations $\rho_i : \phi_i (\pi_1(X_i)) \to U(H)$.

Since $\rho_0$ is faithful, the kernels $H_i : = Ker (\rho_i)$ form a sequence of small subgroups of measure preserving isometries of $Y_i$. From the splitting $\tilde{X}_i = Y_i \times \mathbb{R}^m$, one obtains inclusions $H_i \to  Iso (\tilde{X}_i)$ and hence by our no-small-subgroup assumption, the representations $\rho_i$ are faithful for large enough $i$. This implies that $\phi_i (\pi_1(X_i))$ is isomorphic to a subgroup of $U(H)$. The claim then follows from Corollary \ref{dense-abelian}.

Since $Y_i$ is compact, the map $\psi_i$ is proper and the group $\psi_i (\pi_1(X_i))$ is discrete. Hence by Theorem \ref{bieber}, there is $C_2 > 0 $ and abelian subgroups $B_i \leq \psi_i (\pi_1(X_i))$ of index $[\psi_i (\pi_1(X_i)), B_i ]\leq C_2$. Finally, the group 
\[ A_i : =  \left( \phi_i^{-1} (Z_i) \cap \psi_i^{-1}(B_i) \right) \leq \pi_1(X_i) \cap \left( Z_i \times B_i \right)    \]
is abelian of index $[\pi_1 (X_i ) : A_i] \leq C_1C_2$. This contradicts our initial assumption.
\end{proof}

\begin{customthm}{3}
 Let $(X_i,d_i,\mathfrak{m}_i)$ be a sequence of $RCD^{\ast}(0,N;D)$ spaces of rectifiable dimension $ n$ such that the sequence of universal covers $(\tilde{X}_i,\tilde{d}_i,\tilde{\mathfrak{m}}_i)$ converges in the pointed measured Gromov--Hausdorff sense to some $RCD^{\ast}(0,N)$ space of rectifiable dimension $n$. Then there is $C > 0 $ such that for each $i$ there is an abelian subgroup $A_i \leq \pi_1(X_i)$ with index  $[\pi_1(X_i): A_i] \leq C$.    
\end{customthm}

\begin{proof}[Proof of Theorem \ref{ua-1}:]
Combine Theorem \ref{nss} and Theorem \ref{Turing-no-compact}.
\end{proof}

\section{First Betti number}\label{section-6}

To prove Theorem \ref{b1} we require an extension to $RCD^{\ast}(K,N)$ spaces of the normal subgroup theorem by Kapovitch--Wilking. With the tools we have so far our proof is not much different from the one for smooth Riemannian manifolds in \cite{kapovitch-wilking}, but we include it here for completeness and convenience for the reader.

Recall that if $(X,p)$ is a pointed proper metric space, $\Gamma \leq Iso (X)$ is a closed group of isometries, and $r> 0 $, we denote by $\mathcal{G}(\Gamma , X , p , r) \leq \Gamma$ the subgroup generated by the elements $ \gamma \in \Gamma $ with norm $\Vert \gamma \Vert _p  = d( \gamma p , p ) \leq r$. Also, if $X$ is semi-locally-simply-connected, we define its fundamental group $\pi_1(X)$ to be the set of deck transformations of its universal cover.

\begin{thm}\label{normal-thm}
 (Normal Subgroup Theorem) For each $K \in \mathbb{R}$, $N \geq 1$, $D>0$,  $\varepsilon_1 > 0 $, there are $\varepsilon _0 > 0 $, $C_0 \in \mathbb{N}$ such that the following holds. For each $RCD^{\ast}(K,N;D)$ space $(X,d,\mathfrak{m})$, there is $\varepsilon \in [\varepsilon_0, \varepsilon_1]$ and a normal subgroup $G \triangleleft \pi_1(X)$ such that for all $x$ in the universal cover $\tilde{X}$ we have
\begin{itemize}
\item $G \leq \mathcal{G}( \pi_1(X) , \tilde{X}, x  , \varepsilon /1000  )$.
\item $[ \mathcal{G}(\pi_1(X), \tilde{X}, x, \varepsilon): G] \leq C_0$.
\end{itemize}
\end{thm}
For the proof of Theorem \ref{normal-thm} the following elementary observations are required.
\begin{prop}\label{conjugate-generators}
 Let $Y$ be a proper geodesic space, $\Gamma \leq Iso (Y) $ a closed group of isometries,  $y \in Y$, $g \in \Gamma$, and $r > 0 $. Then $\mathcal{G}(\Gamma , Y ,g y, r) = g \cdot \mathcal{G}(\Gamma  , Y , y,r  ) \cdot g^{-1} $. 
\end{prop}
\begin{proof}[Proof of Proposition \ref{conjugate-generators}:]
The result follows immediately from the identity
\[  \Vert ghg^{-1}  \Vert _{ gy }=  \Vert h \Vert _y  \text{ for all }g, h \in  \Gamma   . \]
\end{proof}

\begin{prop}\label{index}
 Let $G$ be a group, $H \triangleleft G$ a normal subgroup, $S\subset G$ a symmetric generating set, and $M \in \mathbb{N}$ such that for any $s_1 , \ldots , s_M \in S$ there are $1 \leq j_0 \leq j_1 \leq M$ with $s_{j_0} \cdots s_{j_1} \in H$. Then $[G:H ] \leq \vert S \vert ^M$ 
\end{prop}

\begin{proof}[Proof of Proposition \ref{index}:]
By hypothesis, each element of $G/H$ can be written as a product of at most $M$ elements of $\{ sH \vert s \in S \}$. There are $\vert S \vert ^M$ such words, so the result follows. 
\end{proof}

\begin{proof}[Proof of Theorem \ref{normal-thm}:]
Working by contradiction, we assume there is a sequence of $RCD^{\ast}(K,N;D)$ spaces $X_i$ and sequences $\delta_i \to 0$ with the property that for any normal subgroup $G \triangleleft \pi_1(X_i)$ and any $\varepsilon  \in  [\delta_i , \varepsilon_1] $ there is $x$ in the universal cover $\tilde{X}_i$ such that either $G$  is not contained in  $\mathcal{G}(\pi_1(X_i) , \tilde{X}_i, x  , \varepsilon /1000  )$ or the index of $G$ in $  \mathcal{G}(\pi_1(X_i), \tilde{X}_i, x, \varepsilon)$ is greater than $i$. 

After choosing $p_i \in \tilde{X}_i$ and passing to a subsequence, we can assume $(\tilde{X}_i , p_i)$ converges in the pmGH sense to a pointed $RCD^{\ast}(K,N)$ space $(X,p)$, and $\pi_1(X_i)$ converges equivariantly to a closed group $\Gamma \leq Iso(X)$. By Corollary \ref{clean-uc}, the discrete approximation $f_i : \pi_1(X_i) \to \Gamma$ constructed in Section \ref{reformulations} is clean.

Let $r \in ( 0  , 1/2D ]  $ be such that $B_e(r, \Gamma)$ is contained in the connected component of the identity in $\Gamma$, and set $G_i \leq \pi_1(X_i)$ as the subgroup generated by $ f_i ^{-1}(B_e(r,\Gamma ))$. By Lemma \ref{clean-cci}, for $i$ large enough, $G_i $ is a normal subgroup of $\pi_1(X_i)$. Also, by Proposition \ref{discrete-gh}, if $\delta  >0 $ is small enough, for large enough $i$ we have
\begin{itemize}
\item $f_i (B_e(\delta , \pi_1(X_i))) \subset B_e(r, \Gamma ).$
\item $f_i^{-1}(B_e(\delta /2 , \Gamma )) \subset B_e(\delta , \pi_1(X_i) )$.
\end{itemize}  
From the proof of Lemma  \ref{cci}, this implies that if $i$ is large enough, then 
\[  G_i  = \left\langle   B_e(\delta , \pi_1(X_i) )   \right\rangle     .      \] 
This implies, using Equation \ref{d0}, that
\[  G_i \leq \bigcap _{x \in B_{p_i}(1/\delta , \tilde{X}_i)} \mathcal{G}( \pi_1(X_i)  , \tilde{X}_i , x , \delta ).                \]
If $\delta < 1/2D$, then any element $x$ in $\tilde{X}_i$ is sent by some $g \in \pi_1(X_i)$ to an element $gx$ of $B_{p_i}(1/\delta , \tilde{X}_i)$, and from the fact that $G_i$ is normal and Proposition  \ref{conjugate-generators} we have 
\[G_i = g^{-1} G_i g \subset g^{-1} \cdot \mathcal{G}(\pi_1(X_i) , \tilde{X}_i, gx, \delta ) \cdot g =  \mathcal{G}(\pi_1(X_i) , \tilde{X}_i, x, \delta ).  \]
The contradiction is then attained with the following claim by setting $\delta = \varepsilon /1000$.
\begin{center}
\textbf{Claim: }There is $C_0 > 0 $ such that if $\varepsilon > 0 $ is small enough, then for $i$ large enough, $[\mathcal{G}(\pi_1(X_i), \tilde{X}_i, x, \varepsilon ) : G_i  ] \leq C_0$ for all $x \in \tilde{X}_i$.
\end{center}
By Theorem \ref{bg-in}, there is $m \in \mathbb{N}$ such that any ball of radius $10/r$ in an $RCD^{\ast}(K,N)$ space can be covered by $m$ balls of radius $r/10$. Set $\varepsilon : = r/10m^m$ and  $M: = m^m $.  Then for any $x  \in \tilde{X}_i$ and elements $g_1 , \ldots , g_M   \in  \{ g \in  \pi_1(X_i) \vert \Vert g \Vert _x \leq  \varepsilon \}$, by the pigeonhole principle, there are $1 \leq j_0 \leq j_1 \leq M$ with 
\[ g_{j_0} \cdots g_{j_1} \in \{ g \in \pi_1(X_i) \vert d(gy,y) \leq r/2 \text{ for all }y \in B_x(2/r, \tilde{X}_i) \} \subset G_i  .  \] 
By Theorem \ref{short-short}, there is a set $S \subset \{ g \in  \pi_1(X_i) \vert \Vert g \Vert _x \leq  \varepsilon \} $ with $\langle S \rangle = \mathcal{G}(\pi_1(X_i), \tilde{X}_i , x, \varepsilon )$ of size $\vert S \vert \leq C ( k, N, D  )$, so by Proposition \ref{index}, we get
\[ [\mathcal{G}(\pi_1(X_i), \tilde{X}_i, x, \varepsilon ) : G_i  ] \leq C_0 : = C^M  .\] 
\end{proof}

For a geodesic space $X$ and a collection $\mathcal{U}$ of subsets of $X$, we denote by $H_1^{\mathcal{U}}(X)$ the subgroup of the first homology group $H_1(X)$ generated by the images of the maps $H_1(U) \to H_1(X)$ induced by the inclusions $U \to X$  with $U \in \mathcal{U}$. If $\mathcal{U}$ is the collection of balls of radius $\delta$ in $X$, we denote $H^{\mathcal{U}}_1(X)$ by $H_1^{\delta}(X)$. This group satisfies a natural monotonicity property.

\begin{lem}\label{refine}
 Let $X$ be a geodesic space and $\mathcal{U}, \mathcal{V} $ two families of subsets of $X$. If for each $U \in \mathcal{U}$, there is $V \in \mathcal{V}$ with $U \subset V$, then $H_1^{\mathcal{U}}(X) \leq H_1^{\mathcal{V}}(X)$.
\end{lem}

\begin{defn}
\rm For metric spaces $X$ and $Y$, a function  $\phi : Y \to X$ is said to be  an $\varepsilon$-isometry if  
\begin{itemize}
\item For each $y_1, y_2 \in Y$, one has $\vert d(fy_1, fy_2)-d(y_1, y_2 ) \vert \leq \varepsilon $.
\item For each $x \in X$, there is $y \in Y$ with $d(fy,x) \leq \varepsilon$.
\end{itemize}
\end{defn}

\begin{thm}\label{sw}
 (Sormani--Wei) Let $X$ be a compact geodesic space and assume there is  $\varepsilon _0 >0$ such that $H_1^{\varepsilon _0}(X) $ is trivial. If $Y$ is a  compact geodesic space and  $f : Y \to X$ a $\delta $-isometry with $\delta \leq \varepsilon_0 /100$, then there is a surjective morphism 
\[  \tilde{\phi } : H_1(Y) \to H_1(X)        \]
whose kernel is  $H_1^{10 \delta }(Y) =  H_1^{\varepsilon_0 /10}(Y)$.
\end{thm}

\begin{proof}
We follow the lines of (\cite{sormani-wei}, Theorem 3.4), where they prove this result for $\pi_1$ instead of $H_1$. Each 1-cycle in $Y$ can be thought as a family of loops $\mathbb{S}^1 \to Y$ with integer multiplicity. For each map $ \gamma : \mathbb{S}^1 \to Y$, by uniform continuity one could pick finitely many cyclically ordered points $\{ z_1, \ldots , z_m \} \subset \mathbb{S}^1$ such that $\gamma ([z_{j-1}, z_j])$ is contained in a ball of radius $\varepsilon _0  /10$ for each $j$. Then set $\phi ( \gamma ) : \mathbb{S}^1 \to X $ to be the loop with $\phi (\gamma )(z_j) = f( \gamma (z_j) )$ for each $j$, and $\phi (\gamma ) \vert _{[z_{j-1}, z_j]}$ a minimizing geodesic from $\phi (\gamma )(z_{j-1})$ to $\phi (\gamma )(z_{j})$. 

Clearly, $\phi (\gamma )$ depends on the choice of the points  $z_j$ and the minimizing paths $\phi (\gamma ) \vert _{[z_{j-1}, z_j]}$. However, the homology class of $\phi (\gamma )$ in $H_1(X)$ does not depend on these choices, since different choices yield curves that are $\varepsilon_0/2$-uniformly close, which by hypothesis are homologous. 

Assume that a 1-cycle $c$ in $Y$ is the boundary $\partial \sigma$ of a 2-chain $\sigma$. After taking iterated barycentric subdivision, one could assume that each simplex of $\sigma$ is contained in a ball of radius $\varepsilon _0 /10$. By recreating $\sigma $ in $X$ via $f$ simplex by simplex, one could find a 2-chain whose boundary is $\phi (c)$. This means that $\phi$ induces a map $ \tilde{\phi}:H_1(Y) \to H_1(X)$.

In a similar fashion, if a 1-cycle $c$ in $Y$ is such that $\phi (c) $ is the boundary of a 2-chain $\sigma$, one could again apply iterated barycentric subdivision to obtain a 2-chain $\sigma^{\prime}$ in $X$ whose boundary is $\phi (c)$ and such that each simplex is contained in a ball of radius $\delta$. Using $f$ one could recreate the 1-skeleton of $\sigma ^{\prime}$ in $Y$ in such a way that expresses $c$ as a linear combination with integer coefficients of 1-cycles contained in balls of radius $10\delta$ in $Y$. This implies that the kernel of $\tilde{\phi} $ is contained in $H_1^{10 \delta}(Y)$.

If a 1-cycle $c$ in $Y$ is contained in a ball of radius $\varepsilon _ 0 /10$, then $\phi (c)$ is contained in a ball of radius $\varepsilon _ 0 /2$ and then by hypothesis, $\phi (c)$ is a boundary. Together with the previous paragraph, this shows that the kernel of $\tilde{\phi}$ is precisely $H_1^{10 \delta}(Y) = H_1^{\varepsilon_0/10}(Y)$.

Lastly, for any loop $\gamma : \mathbb{S}^1 \to X$, one can create via $f$ a loop  $\gamma_1 : \mathbb{S}^1 \to Y$ such that $\phi ( \gamma_1)$ is $10 \delta$-uniformly close (and hence homologous) to $\gamma$, so $\tilde{\phi}$ is surjective. 
\end{proof}

Clearly, a sequence of compact metric spaces $X_i$ converges in the GH sense to a compact metric space $X$ if and only if there is a sequence of $\delta_i$-isometries $\phi_i : X_i \to X$ with $\delta_i \to 0$. Hence by Theorem \ref{sw} one gets the following.

\begin{cor}\label{sw-gap}
 Let $X$ be a pointed compact geodesic space and assume there is  $\varepsilon _0 >0$ such that $H_1^{\varepsilon _0}(X) $ is trivial. If a sequence of compact geodesic spaces $X_i$ converges to $X$ in the GH sense, then there is a sequence $\rho_i \to 0$ such that $H_1^{\rho_i}(X_i)= H_1^{\varepsilon_0/10}(X_i)$  for all $i$.
\end{cor}

For the proof of Theorem \ref{b1}, it is convenient to reformulate Theorem \ref{normal-thm} in the abelian setting.

\begin{thm}\label{normal-abelian}
  For each $K \in \mathbb{R}$, $N \geq 1$, $D > 0 $, $\varepsilon_1 > 0 $, there are $\varepsilon_0$, $C_0 \in \mathbb{N}$ such that the following holds. For each $RCD^{\ast}(K,N;D)$ space $(X,d,\mathfrak{m})$, there is $\varepsilon \in [\varepsilon_0 , \varepsilon_1]$ and a subgroup $G \leq H_1(X)$ such that for all $x$ in the covering space $X^{\prime} \to X$ with Galois group $H_1(X)$ one has
\begin{itemize}
\item $G \leq \mathcal{G} (H_1(X), X^{\prime}, x, \varepsilon / 1000)  $.
\item $[\mathcal{G}(H_1(X), X^{\prime}, x, \varepsilon ): G] \leq C_0$.
\end{itemize}
\end{thm}

\begin{customthm}{5}
 Let $(X_i,d_i,\mathfrak{m}_i)$ be $RCD^{\ast}(K,N;D)$ spaces of rectifiable dimension $ n$ and first Betti number $\beta_1(X_i) \geq r$. If the sequence $X_i$ converges in the measured Gromov--Hausdorff sense to an $RCD^{\ast}(K,N;D)$ space $X$ of rectifiable dimension $m$, then
\[ \beta_1(X) \geq r + m - n .\]     
\end{customthm}

\begin{proof}[Proof of Theorem \ref{b1}:] 
 Let $p \in X$ be an $m$-regular point and choose $p_i \in X_i$ converging to $p$. For each $i$, let $X_i^{\prime} : = \tilde{X}_i / [\pi_1(X_i),\pi_1(X_i)]$ denote the regular cover of $X_i$ with Galois group $H_1(X_i)$, and choose $q_i \in X_i^{\prime}$ in the preimage of $p_i$. Then by Lemma \ref{gap}, there is $\eta > 0 $ and a sequence $\eta_i \to 0$ such that  
\[ \mathcal{G} ( H_1(X_i), X_i^{\prime} , q_i , \eta_i) = \mathcal{G} ( H_1(X_i), X_i^{\prime} , q_i , \eta) \text{ for all }i.        \]
By Theorem  \ref{jikang}, there is  $ \varepsilon_1 \in (0, \eta ] $ such that $ H^{\varepsilon_1}_1(X)$ is trivial. By Theorem \ref{sw},  all we need to show is that for $i$ large enough, $H_1^{\varepsilon_1/10}(X_i)$ has rank $\leq n -m $. By Corollary \ref{sw-gap}, there is a sequence $\rho_i \to 0$ with $H_1^{\rho_i}(X_i) = H_1^{\varepsilon_1/10}(X_i)$ for all $i$.   By Theorem \ref{normal-abelian}, there are $\varepsilon _ 0 > 0 $, $C_0 \in \mathbb{N}$, subgroups $G_i \leq H_1(X_i)$, and a sequence $\delta_i \in [3 \varepsilon _ 0 , \varepsilon _1 /10] $ with the property that for each $x \in X_i^{\prime}$, 
\begin{center}
$\mathcal{G}( H_1(X_i),  X_i^{\prime}, x, \delta_i)$ contains $G_i$ as a subgroup of index $\leq C_0$.
\end{center}
Let $ x_1, \ldots , x_m  \in X$ be such that 
\[X = \bigcup_{j=1}^m B_{x_j}( \varepsilon_0/3,X),\]
 and for each $j$, choose $x^i_j \in X_i$ converging to $x_j$. Then for large enough $i$, the balls $B_{x^i_j}( \varepsilon_0/2, X_i)$ cover $X_i$. This implies that for large enough $i$, each ball of radius $\rho_i$ in $X_i$ is contained in a ball of the form $B_{x^i_j}( \varepsilon_0, X_i)$. Hence if we let $\mathcal{U}_i$ denote the family $\{ B_{x_j^i}(  \delta_i /3 , X_i )   \}_{j=1}^m$ for each $i$,  by Lemma \ref{refine} we get
\[ H_1^{\rho_i}(X_i) \leq  H_1^{ \mathcal{U}_i}( X_i ) \leq  H_1^{\varepsilon_1/10}(X_i) = H_1^{\rho_i} (X_i) .      \]
If we choose $y_j^i \in X_i^{\prime}$ in the preimage of $x_j^i$, then by construction of $H_1^{\mathcal{U}_i}(X_i)$ we have 
\[ H_1^{\mathcal{U}_i} (X_i) \leq  \left\langle  \bigcup_{j=1}^m \mathcal{G}(H_1(X_i), X_i^{\prime}, y_j^i ,  \delta_i  )       \right\rangle .    \]
 Since $H_1^{\mathcal{U}_i} (X_i)$ is abelian, the index of $G_i $ in $H_1^{\mathcal{U}_i}(X_i)$ is at most $C_0^m$. Therefore, the rank of $H_1^{\varepsilon_1 /10}(X_i)$ equals the rank of $G_i$ for $i$ large enough. Set $\Gamma_i  : =  \mathcal{G}( H_1(X_i), X_i^{\prime}, q_i , \eta ) $. 
 
 Since $\eta \geq \varepsilon_1$,  $\Gamma_i$ contains $G_i$, and since $\Gamma_i = \mathcal{G}( H_1(X_i), X_i^{\prime}, q_i , \eta_i )  $ and $\eta_i \leq \varepsilon_0$ for large enough $i$, the index of $G_i$ in $\Gamma_i$ is finite. Hence the following claim implies Theorem \ref{b1}. 

\begin{center}
\textbf{Claim:} For $i$ large enough, $\Gamma_i$ has rank $\leq n - m$.
\end{center}
Let $\lambda_i \to \infty$ be a sequence that diverges so slowly that $\lambda_i \eta_i \to 0$ and the sequence $(\lambda_i X_i , p_i)$ converges in the pGH sense to $(\mathbb{R}^{m}, 0)$. We can achieve this since $p$ is $m$-regular and $\eta_ i \to 0$. After taking a subsequence, we can assume the sequence $(\lambda_i X_i^{\prime}, q_i)$ converges in the pmGH sense to an $RCD^{\ast}(K,N)$ space $(Y,q)$, and the groups $H_1(X_i)$ converge equivariantly to some group $\Gamma \leq Iso (Y)$. Since all elements of $H_1(X_i) \backslash \Gamma_i$ move $q_i$ at least $\eta \lambda_i$ away, $\Gamma_i$ converges equivariantly to $\Gamma$ as well. From the definition of equivariant convergence, it follows that the $\Gamma_i$-orbits of $q_i$ converge in the pGH sense to the $\Gamma$-orbit of $q$. 

By Proposition \ref{gigli-corollary}, $Y$ splits isometrically as a product $\mathbb{R}^{m} \times Z$ with $Z$ an $RCD^{\ast}(0, N- m)$ space of rectifiable dimension $\leq n-m$, such that the $Z$-fibers coincide with the $\Gamma$-orbits. By Remark \ref{homogeneous}, $Z$ has topological dimension $\leq n - m$, and by Theorem \ref{zamora-1}, the rank of $\Gamma_i$ is at most $n - m$ for large enough $i$.
\end{proof} 

\noindent \textbf{Acknowledgments.} The authors would like to thank Andrea Mondino, Raquel Perales, Anton Petrunin, Guofang Wei for interesting discussions and helpful comments, and an anonymous referee for comments and suggestions that improved the presentation of the paper. Both authors were supported by Postdoctoral Research Fellowships granted by the Max Planck Institute for Mathematics in Bonn.

\end{document}